\documentclass[a4paper,11pt,reqno]{amsart}

\usepackage[utf8]{inputenc}
\usepackage[T1]{fontenc}
\usepackage{latexsym,amsmath,amsfonts,amscd,amssymb,amsthm,mathrsfs}
\usepackage{pifont}
\usepackage{eucal}
\usepackage{enumerate}
\usepackage[all]{xy}
\usepackage{tikz}

\definecolor{cobalt}{RGB}{61,89,171}
\usepackage[colorlinks,citecolor=cobalt,linkcolor=cobalt,urlcolor=cobalt,pdfpagemode=UseNone,backref = page]{hyperref}

\newcommand{\exc}{\mathsf{E}}

\newcommand{\del}{\partial}
\newcommand{\delbar}{\overline\del}
\newcommand{\Z}{\mathbb{Z}}

\newcommand{\C}{\mathbb{C}}

\newcommand{\op}{\operatorname}
\newcommand{\iso}{isomorphism}
\newcommand{\ra}{\rightarrow}
\newcommand{\lra}{\longrightarrow}
\newcommand{\st}{such that}
\newcommand{\str}{structure}

\allowdisplaybreaks[1]

\newcommand{\referenza}{}

\theoremstyle{plain}
\newtheorem{theorem}{Theorem}
\newtheorem*{theorem*}{Theorem \referenza}
\newtheorem{proposition}[theorem]{Proposition}
\newtheorem*{proposition*}{Proposition \referenza}
\newtheorem{lemma}[theorem]{Lemma}
\newtheorem{corollary}[theorem]{Corollary}

\theoremstyle{definition}
\newtheorem{definition}[theorem]{Definition}
\newtheorem{example}[theorem]{Example}
\theoremstyle{remark}
\newtheorem{remark}[theorem]{Remark}
\newtheorem{question}[theorem]{Question}

\begin{document}

\title[Dolbeault cohomology for bimeromorphisms]{Note on Dolbeault cohomology and Hodge structures up to bimeromorphisms}

\author[D. Angella]{Daniele Angella}
\address[Daniele Angella]{
Dipartimento di Matematica e Informatica ``Ulisse Dini''\\
Universit\`a degli Studi di Firenze\\
viale Morgagni 67/a\\
50134 Firenze, Italy
}
\email{daniele.angella@gmail.com}
\email{daniele.angella@unifi.it}

\author[T. Suwa]{Tatsuo Suwa}
\address[Tatsuo Suwa]{
Department of Mathematics\\
Hokkaido University\\
Sapporo 060-0810, Japan
}
\email{tsuwa@sci.hokudai.ac.jp}

\author[N. Tardini]{Nicoletta Tardini}
\address[Nicoletta Tardini]{
Dipartimento di Matematica e Informatica ``Ulisse Dini''\\
Universit\`a degli Studi di Firenze\\
viale Morgagni 67/a\\
50134 Firenze, Italy
}
\email{nicoletta.tardini@gmail.com}

\curraddr{
Dipartimento di Scienze Matematiche, Fisiche e Informatiche\\
Unit\`a di Matematica e Informatica\\
Universit\`a degli Studi di Parma\\
Parco Area delle Scienze 53/A\\
43124 Parma, Italy
}

\author[A. Tomassini]{Adriano Tomassini}
\address[Adriano Tomassini]{
Dipartimento di Scienze Matematiche, Fisiche e Informatiche\\
Unit\`a di Matematica e Informatica\\
Universit\`a degli Studi di Parma\\
Parco Area delle Scienze 53/A\\
43124 Parma, Italy
}
\email{adriano.tomassini@unipr.it}

\keywords{complex manifold, non-K\"ahler geometry, $\partial\overline\partial$-Lemma, Hodge decomposition, modification, blow-up, Dolbeault cohomology, orbifold}
\thanks{During the preparation of this work, the first-named author has been supported by the Project SIR2014 ``Analytic aspects in complex and hypercomplex geometry'' (code RBSI14DYEB), by the Projects PRIN ``Variet\`a reali e complesse: geometria, topologia e analisi armonica' and PRIN2017 ``Real and Complex Manifolds: Topology, Geometry and holomorphic dynamics'' (code 2017JZ2SW5), and by GNSAGA of INdAM.
The second-named author is supported by the JSPS grant no. 16K05116.
The third-named author is supported by Project PRIN ``Variet\`a reali e complesse: geometria, topologia e analisi armonica'', by SIR2014 project RBSI14DYEB ``Analytic aspects in complex and hypercomplex geometry'', and by GNSAGA of INdAM.
The fourth-named author is supported by Project PRIN ``Variet\`a reali e complesse: geometria, topologia e analisi armonica'' and by GNSAGA of INdAM}
\subjclass[2010]{32Q99, 32C35, 32S45}

\date{\today}

\begin{abstract}
We construct a simply-connected compact complex non-K\"ahler manifold satisfying the $\partial\overline\partial$-Lemma, and endowed with a balanced metric.
To this aim, we were initially aimed at investigating the stability of the property of satisfying the $\partial\overline\partial$-Lemma under modifications of compact complex manifolds and orbifolds.
This question has been recently addressed and answered in \cite{rao-yang-yang, yang-yang, stelzig-blowup, stelzig-doublecomplex} with different techniques.
Here, we provide a different approach using \v{C}ech cohomology theory to study the Dolbeault cohomology of the blow-up $\tilde X_Z$ of a compact complex manifold $X$ along a submanifold $Z$ admitting a holomorphically contractible neighbourhood.
\end{abstract}

\maketitle

\section*{Introduction}

The {\em $\partial\overline\partial$-Lemma} is a strong cohomological decomposition property defined for complex manifolds, which is satisfied for example by algebraic projective manifolds and, more generally, by compact K\"ahler manifolds. The property is closely related to the fact that the Dolbeault cohomology provides a Hodge structure on the de Rham cohomology (cf. Subsection \ref{ss-deldelbar} below).

This property yields also strong topological obstructions: the real homotopy type of a compact complex manifold satisfying the $\partial\overline\partial$-Lemma is a formal consequence of its cohomology ring \cite{deligne-griffiths-morgan-sullivan}.
Complex non-K\"ahler manifolds usually do not satisfy the $\partial\overline\partial$-Lemma: for example, it is never satisfied by compact non-tori nilmanifolds \cite{hasegawa}. On the other hand, some examples of compact complex non-K\"ahler manifolds satisfying the $\partial\overline\partial$-Lemma are provided by Moi\v{s}hezon manifolds and manifolds in class $\mathcal{C}$ of Fujiki thanks to \cite[Theorem 5.22]{deligne-griffiths-morgan-sullivan}, see \cite{hironaka} for a concrete example.
By the results contained in \cite[Corollary 3.13]{campana}, \cite[Theorem 1]{lebrun-poon} and thanks to the stability property of the $\partial\overline\partial$-Lemma for small deformations \cite[Proposition 9.21]{voisin-1}, \cite[Theorem 5.12]{wu} one can produce examples of compact complex manifolds satisfying the
$\partial\overline\partial$-Lemma and not bimeromorphic to K\"ahler manifolds. Other examples of this kind can be found among solvmanifolds \cite{angella-kasuya-AGAG, angella-kasuya-NWEJM, kasuya-JGP}; moreover other examples are provided by Clemens manifolds \cite{friedman, friedman-arxiv1708}, which are constructed by combining modifications and deformations.

The main aim of this note is to construct a simply-connected compact complex non-K\"ahler manifold satisfying the $\partial\overline\partial$-Lemma.

The theorem in \cite[Theorem 5.22]{deligne-griffiths-morgan-sullivan} states that, for a modification $\tilde X\to X$ of compact complex manifolds, the property of $\partial\overline\partial$-Lemma is preserved from $\tilde X$ to $X$. So, it is natural to ask whether it is in fact an invariant property by modifications. This is true, for example, for compact complex surfaces, thanks to the topological Lamari's and Buchdahl's criterion \cite{lamari, buchdahl}.
Note that, in higher dimension, the K\"ahler property is not stable under modifications; but there are weaker metric properties that are, for example the {\em balanced} condition in the sense of Michelsohn \cite[Corollary 5.7]{alessandrini-bassanelli} or the {\em strongly-Gauduchon} condition in the sense of Popovici \cite[Theorem 1.3]{popovici-JGA}. In fact, it is conjectured that the metric balanced condition and the cohomological $\partial\overline\partial$-Lemma property are strictly related to each other, see for example \cite[Conjecture 6.1]{popovici}, see also \cite{tosatti-weinkove-Crelle, popovici-BSMF}; and this provides another motivation for the above question.

\medskip

In this note, we deal with the Dolbeault cohomology of the blow-up along submanifolds. The strategy we follow is sheaf-theoretic, more precisely \v{C}ech-cohomological, in the spirit of \cite{suwa-ASPM}.
The de Rham case in the K\"ahler context is considered in \cite[Theorem 7.31]{voisin-1}.
For our argument, we need to assume that {\em the centre admits a holomorphically contractible neighbourhood} (this is clearly satisfied when blowing-up at a point, see also the explicit computations in Example \ref{example:blowup-point}) and another technical assumption \eqref{eq:technical} concerning the kernel and images of certain morphisms.
We can then deduce that:

\renewcommand{\referenza}{\ref{thm:blowup}}
\begin{theorem*}
Let $X$ be a compact complex manifold and $Z$ a closed submanifold of $X$.
If both $X$ and the centre $Z$ admits a Hodge structure {\rm (}in the sense of Definition~\ref{defH}{\rm)}, then the same holds for the blow-up $\mathrm{Bl}_ZX$ of $X$ along $Z$, provided that $Z$ admits a holomorphically contractible neighbourhood and the technical assumption \eqref{eq:technical} holds.
\end{theorem*}

Along the way we give explicit expressions for the de Rham and Dolbeault cohomologies of
$\mathrm{Bl}_ZX$ (see Propositions \ref{isodR} and \ref{isoDol}).

Hopefully, a further study of the cohomological properties of submanifolds (see Question \ref{question:submanifold}) and a deeper use of techniques as the MacPherson's deformation to the normal cone (see Question \ref{question:normal-cone}), along with the Weak Factorization Theorem for bimeromorphic maps in the complex-analytic category \cite[Theorem 0.3.1]{akmw}, \cite{wlodarczyk}, may allow to use the above techniques to prove in full generality the stability of the $\partial\overline\partial$-Lemma under modifications, see Remark \ref{remark:corollary}.

During the preparation of this work, several other attempts to solve the same problems appeared \cite{rao-yang-yang, yang-yang, stelzig-blowup}, using different techniques. In particular, the work by Jonas Stelzig \cite{stelzig-thesis} finally ties up the problem, as far as now:

\begin{theorem}[{\cite[Theorem 8]{stelzig-blowup}, \cite[Corollary 25]{stelzig-doublecomplex}}]\label{thm:stelzig}
The $\partial\overline\partial$-Lemma property is a bimeromorphic invariant if and only if it is invariant by restriction.
\end{theorem}

Even if Stelzig's theorem is clearly stronger than our Theorem \ref{thm:blowup}, we think that our argument may be interesting and useful in providing a broader point of view for understanding (\v{C}ech-)Dolbeault cohomology.

\medskip

The second and main aim of this note is to construct new explicit examples of compact complex manifolds satisfying the $\partial\overline\partial$-Lemma: in particular, we provide a {\em simply-connected example}, see Example \ref{example:global-quotient-iwasawa}.
To this aim, we need to work with {\em orbifolds} in the sense of Satake \cite{satake}, and their desingularizations. We take advantage of Stelzig's general results, see Theorem \ref{thm:orbifolds}.
The construction of Example \ref{example:global-quotient-iwasawa} goes as follows, see {\itshape e.g.}  \cite{fernandez-munoz, bazzoni-fernandez-munoz}: we start from a manifold isomorphic to the Iwasawa manifold, which does not satisfy the $\partial\overline\partial$-Lemma; then we quotient it by a finite group of automorphisms; and then we resolve its singularities.
Finally, by Theorem \ref{thm:orbifolds}, we get simply-connected examples of complex manifolds satisfying the $\partial\overline\partial$-Lemma:

\renewcommand{\referenza}{\ref{thm:examples-simply-connected}}
\begin{theorem*}
There exist a {\em simply-connected} compact complex non-K\"ahler manifold, (not even in class $\mathcal C$ of Fujiki,) that satisfy the $\partial\overline\partial$-Lemma. Our example admits a balanced metric.
\end{theorem*}

As far as we know, these are the first explicit examples of simply-connected compact complex non-K\"ahler manifolds satisfying the $\del\delbar$-Lemma in the literature.

\bigskip

{\small
\noindent{\sl Acknowledgments.}
The authors warmly thank Giovanni Bazzoni, Carlo Collari, Hisashi Kasuya, Sheng Rao, Jonas Stelzig, Valentino Tosatti, Claire Voisin, Victor Vuletescu, Song Yang, Xiangdong Yang, Weiyi Zhang for interesting and useful discussions on many occasions.
Thanks also to the anonymous Referees for some suggestions and comments that improve the presentation of the paper.
Parts of this work have been written during the visit of the first-named author at the Hokkaido University and that of the second-named author at Universit\`a di Firenze: they would like to thank the both Departments for the warm hospitality.
}

\section{Preliminaries on \v{C}ech-Dolbeault cohomology and $\del\delbar$-Lemma}
In this Section, we recall the main definitions and results about relative
\v{C}ech-de Rham and \v{C}ech-Dolbeault cohomologies; for more details and applications we refer to \cite{suwa-book} and \cite{suwa-ASPM}. We also recall the $\del\delbar$-Lemma and some of its characterizations.

\subsection{\v{C}ech-de Rham cohomology and relative de Rham cohomology}
Let $X$ be a smooth manifold.
Let $\mathcal{U}=\{U_{0},U_{1}\}$ be an open covering of $X$ and set $U_{01}:=U_0\cap U_1$. Denoting by $A^{h}(U)$ the space of ($\mathbb C$-valued) smooth $h$-forms on an open set $U$ in $X$, we set 
\[
A^{h}(\mathcal{U}):=A^{h}(U_{0})\oplus A^{h}(U_{1})
\oplus A^{h-1}(U_{01}).
\]
The differential operator $D\colon A^{h}(\mathcal{U})\to A^{h+1}(\mathcal{U})$ defined by $D\left(\sigma_{0},\,\sigma_{1},\,\sigma_{01}\right)=\left(d\sigma_{0},\,d\sigma_{1},\,
\sigma_{1}-\sigma_{0}-d\sigma_{01}\right)$ yields a differential complex $(A^{\bullet}(\mathcal{U}),D)$: the {\em \v{C}ech-de Rham cohomology} associated to the covering $\mathcal{U}$
is then defined by $H^{\bullet}_{D}(\mathcal{U})
=\ker D/\mathrm{im}\,D$.
The 
morphism $A^{h}(X)\to A^{h}(\mathcal{U})$ given by $\omega\mapsto (\omega|_{U_{0}},\omega|_{U_{1}},0)$ induces an isomorphism in cohomology, \cite[Theorem 3.3]{suwa-book},
\[
H^{\bullet}_{dR}(X)\overset\sim\lra H^{\bullet}_{D}(\mathcal{U}),
\]
whose inverse is given by assigning to the class of $\sigma=(\sigma_{0},\sigma_{1},\sigma_{01})$ the class of the global $d$-closed form $\rho_{0}\sigma_{0}+\rho_{1}\sigma_{1}-d\rho_{0}\wedge\sigma_{01}$, where $(\rho_{0},\rho_{1})$ is a partition of unity subordinate to $\mathcal{U}$. In the above $H^{\bullet}_{dR}(X)$ denotes the de~Rham cohomology of $X$. The de Rham theorem says it is isomorphic to $H^{\bullet}(X;\C)$, the simplicial, singular or sheaf cohomology of $X$ with coefficents in $\C$.
See \cite{suwa-book} for further results, including cup product, integration on top-degree cohomology, duality.

Given a closed set $S$ in $X$, we can take $U_0:=X\setminus S$ and $U_1$ an open neighbourhood of $S$ in $X$, and the open covering $\mathcal U=\{U_0,U_1\}$. In this case, define 
$A^p(\mathcal U, U_0):=\{\,\sigma\in A^{h}(\mathcal{U})\mid\sigma_{0}=0\,\}=A^{h}(U_{1})\oplus A^{h-1}(U_{01})$. Then $\left(A^\bullet(\mathcal U, U_0),D\right)$ is a differential sub-complex 
of 
$\left(A^{\bullet}(\mathcal{U}),D\right)$. 
Let $H^{h}_{D}(\mathcal{U},U_{0})$ denote the associated
cohomology. From the short exact sequence 
\[
0\lra A^{\bullet}(\mathcal{U},U_{0})\lra A^{\bullet}(\mathcal{U})
\lra A^{\bullet}(U_{0})\lra 0,
\]
where the first map is the inclusion and the second map is the projection on the first element, we obtain a long exact sequence in cohomology
\begin{equation}\label{lexcdR}
\cdots\lra H^{h-1}_{dR}(U_{0})\overset\delta\lra H^{h}_{D}(\mathcal{U},U_{0})\overset{j^{*}}\lra H^{h}_{D}(\mathcal{U})\overset{i^{*}}\lra H^{h}_{dR}(U_{0})\lra\cdots.
\end{equation}
From this we see
 that $H^{h}_{D}(\mathcal{U},U_{0})$ is determined uniquely modulo canonical isomorphisms, independently of the choice of $U_{1}$. We denote it also by $H^{h}_D(X,X\setminus S)$ and call it the {\em relative \v{C}ech-de Rham cohomology}. We recall that {\em excision} holds: for any neighbourhood $U$ of $S$ in $X$, it holds $H^{h}_D(X,X\setminus S)\simeq H^{h}_D(U,U\setminus S)$.
In fact we have, \cite{suwa-intersection},
$$ H^{h}_D(X,X\setminus S)\simeq  H^{h}(X,X\setminus S; \mathbb C), $$
the relative cohomology of the pair $(X,X\setminus S)$.

Consider now a smooth complex vector bundle $\pi \colon E \to M$ of rank $k$ on a smooth manifold $M$. Consider the bundle $\varpi \colon \pi^*E \to E$ defined by the fibre product
$$
\SelectTips{cm}{}
 \xymatrix
 @C=.7cm
@R=.7cm
{\pi^*E \ar[r] \ar_{\varpi}[d] & E \ar^{\pi}[d] \\
E \ar[r]^{\pi} & M
} $$
and its diagonal section $s_\Delta$. The zero-set of $s_\Delta$ is the image of the zero-section of $E$, which is identified with $M$. In this
situation, the {\em Thom class} $\Psi_E \in H^{2k}_{D}(E,E\setminus M)$ of $E$ is given as the localization of the top Chern class $c^k(\pi^*E)$ by $s_\Delta$. That is: consider the covering $\mathcal W=\{W_0:=E\setminus M, W_1\}$ of $E$, where $W_1$ is a neighbourhood of $M$ in $E$; consider $\nabla_0$ a connection on $W_0$ such that $\nabla_0 s_\Delta=0$, and $\nabla_1$ a connection on $W_1$; then the Chern class $c^k(\pi^*E)$ is represented by $\left(c^k(\nabla_0),c^k(\nabla_1),c^k(\nabla_{0},\nabla_{1})\right) \in H^{2k}_D(\mathcal W)\simeq H^{2k}_{dR}(E)$, where $c^k(\nabla_{0},\nabla_{1})$ is the Bott difference form of $\nabla_0$ and $\nabla_1$; in fact, since $c^k(\nabla_0)=0$,  this defines a class $\Psi_E\in H^{2k}_{D}(E,E\setminus M)$ represented by $(\psi_1,\psi_{01}):=(c^k(\nabla_1),c^k(\nabla_{0},\nabla_{1}))$.
It turns out that the map
$$
T_E \colon H^{\bullet-2k}_{dR}(M) \stackrel{\sim}{\lra} H^\bullet_{D}(E,E\setminus M) ,
\qquad [\theta]\mapsto\Psi_E\smallsmile\pi^*[\theta]
$$
is an isomorphisms, \cite[Theorem 5.3]{suwa-book}, called the {\em Thom isomorphism}, where the cup
product $\Psi_E\smallsmile\pi^*[\theta]$ is represented by $(\psi_1\wedge \pi^*\theta, \psi_{01}\wedge\pi^*\theta)$. Its inverse is the {\em integration along the fibres}:
$$ \pi_* \colon H^\bullet_{D}(E,E\setminus M) \lra H^{\bullet-2k}_{dR}(M),
\qquad
\pi_*\left(\sigma_1,\sigma_{01}\right) = (\pi_1)_*\sigma_1 + (\pi_{01})_*\sigma_{01} ,$$
where $\pi_1$ is the restriction of $\pi$ to a bundle $T_1$ of disks of complex dimension $k$ in $W_1$, and $\pi_{01}$ is the restriction of $\pi$ to the bundle $T_{01}=-\partial T_1$ of spheres of real dimension $2k-1$ with opposite orientation. In particular, the Thom class $\Psi_E$ is characterized in $H^{2k}_{D}(E,E\setminus M)$ by the property $\pi_*\Psi_E=1$. Finally, we recall the {\em projection formula}, \cite[Ch.II, Proposition 5.1]{suwa-book}: for $\sigma\in A^p(\mathcal{W},W_{0})$, $\theta\in A^q(M)$,
$$ \pi_*(\sigma\smallsmile\pi^*\theta)=\pi_*\sigma\wedge\theta. $$

Given a closed complex submanifold $Z$, of complex codimension $k$, of a complex manifold $X$, of complex dimension $n$, we can define the Thom isomorphism and the Thom class of $Z$ as follows. Consider the normal bundle $\pi \colon N_{Z|X} \to Z$, of complex rank $k$. By the Tubular Neighbourhood Theorem, there exist neighbourhoods $U$ of $Z$ in $X$, and $W$ of $Z$ as zero section in $N_{Z|X}$, and a smooth diffeomorphism $\varphi \colon U \to W$ such that $\varphi|_Z=\mathrm{id}$. Then, setting $N=N_{Z|X}$, we get isomorphisms 
\[
H^\bullet_{D}(X,X\setminus Z)\simeq H^\bullet_D(U,U\setminus Z) \underset{\varphi^{*}}{\overset\sim\longleftarrow} H^\bullet_D(W,W\setminus Z)\simeq H^\bullet_D(N,N\setminus Z). 
\]
Define the Thom class $\Psi_Z\in H^{2k}_{D}(X,X\setminus Z)$ of $Z$ as the image of $\Psi_{N_{Z|X}}$ via the above isomorphisms, and the Thom isomorphism $T_Z \colon H^{\bullet-2k}_D(Z) \stackrel{\sim}{\to} H^\bullet_D(X,X\setminus Z)$ as $T_Z(z)=\Psi_Z\smallsmile r^*z$, where $r=\pi\circ\varphi\colon U \to Z$.

\subsection{\v{C}ech-Dolbeault cohomology} 
Let $X$ be a complex manifold
and let $A^{p,q}(U)$ be the space of smooth $(p,q)$-forms on an open set $U$ in $X$. 
Let $\mathcal{U}=\{U_{0},U_{1}\}$ be an open covering of $X$ and consider
\[
A^{p,q}(\mathcal{U}):=A^{p,q}(U_{0})\oplus A^{p,q}(U_{1})
\oplus A^{p,q-1}(U_{01}).
\]
The differential operator $\bar D\colon A^{p,q}(\mathcal{U})\to A^{p,q+1}(\mathcal{U})$ is defined on every element 
$(\xi_{0},\,\xi_{1},\,\xi_{01})\in A^{p,q}(\mathcal{U})$ by 
$$
\bar D\left(\xi_{0},\,\xi_{1},\,\xi_{01}\right)=\left(\delbar\xi_{0},\,\delbar\xi_{1},\,
\xi_{1}-\xi_{0}-\delbar\xi_{01}\right).
$$
The {\em \v{C}ech-Dolbeault cohomology} associated to the covering $\mathcal{U}$
is then defined by $H^{\bullet,\bullet}_{\bar D}(\mathcal{U})
=\ker\bar D/\mathrm{im}\,\bar D$
(see \cite{suwa-ASPM} where this definition is given for an arbitrary open covering of the manifold $X$).
The 
morphism $A^{p,q}(X)\to A^{p,q}(\mathcal{U})$ given by $\omega\mapsto (\omega|_{U_{0}}\omega|_{U_{1}},0)$ induces an isomorphism in cohomology
\[
H^{\bullet,\bullet}_{\delbar}(X)\overset\sim\lra H^{\bullet,\bullet}_{\bar D}(\mathcal{U}),
\]
where $H^{\bullet,\bullet}_{\delbar}(X)$ denotes the Dolbeault cohomology of $X$, \cite[Theorem~1.2]{suwa-ASPM}.
In particular, the definition is independent of the choice of the covering of $X$.
Moreover, the inverse map is given by assigning to the class of $\xi=(\xi_{0},\xi_{1},\xi_{01})$ the class
of the global $\delbar$-closed form $\rho_{0}\xi_{0}+\rho_{1}\xi_{1}-\delbar\rho_{0}\wedge\xi_{01}$, where $(\rho_{0},\rho_{1})$
is a partition of unity subordinate to  $\mathcal{U}$.

One can define cup product, integration on top-degree cohomology and Kodaira-Serre duality and they turn out to be compatible with the above isomorphism (cf. \cite{suwa-ASPM} for more details).

\subsection{Relative \v{C}ech-Dolbeault cohomology}
Let $S$ be a closed set in $X$. We set $U_{0}=X\setminus S$ and $U_{1}$ to be an open neighbourhood of $S$ in $X$, and we consider the associated covering $\mathcal{U}=\{U_{0},U_{1}\}$ of $X$. For any $p,q$, we set
\[
A^{p,q}(\mathcal{U},U_{0}):=\{\,\xi\in A^{p,q}(\mathcal{U})\mid\xi_{0}=0\,\}=A^{p,q}(U_{1})\oplus A^{p,q-1}(U_{01}).
\]
Then $\left(A^{p,\bullet}(\mathcal{U},U_{0}), \,\bar D\right)$ is a subcomplex of 
$\left(A^{p,\bullet}(\mathcal{U}),\,\bar D\right)$. 
Let $H^{p,q}_{\bar D}(\mathcal{U},U_{0})$ be the
cohomology associated to $\left(A^{p,\bullet}(\mathcal{U},U_{0}),\,\bar D\right)$. From the short exact sequence 
\[
0\lra A^{p,\bullet}(\mathcal{U},U_{0})\lra A^{p,\bullet}(\mathcal{U})
\lra A^{p,\bullet}(U_{0})\lra 0,
\]
where the first map is the inclusion and the second map is the projection on the first element, we obtain a long exact sequence in cohomology
\begin{equation}\label{lexcd}
\cdots\lra H^{p,q-1}_{\delbar}(U_{0})\overset\delta\lra H^{p,q}_{\bar D}(\mathcal{U},U_{0})\overset{j^{*}}\lra H^{p,q}_{\bar D}(\mathcal{U})\overset{i^{*}}\lra H^{p,q}_{\delbar}(U_{0})\lra\cdots.
\end{equation}
Therefore, $H^{\bullet,\bullet}_{\bar D}(\mathcal{U},U_{0})$ is determined uniquely modulo
canonical isomorphism, independently of the choice of $U_{1}$. We denote it also by  $H^{\bullet,\bullet}_{\bar D}(X,X\setminus S)$
and we call it the \emph{relative \v{C}ech-Dolbeault cohomology} of $X$, see \cite[Section 2]{suwa-ASPM}, where it is denoted by $H^{\bullet,\bullet}_{\overline \partial}(X,X\setminus S)$.
We recall that {\em excision} holds: for any neighbourhood $U$ of $S$ in $X$, it holds $H^{\bullet,\bullet}_{\bar D}(X,X\setminus S)\simeq H^{\bullet,\bullet}_{\bar D}(U,U\setminus S)$. 
In fact we have, \cite{suwa-relD},
\[
H^{p,q}_{\bar D}(X,X\setminus S)\simeq  H^q(X,X\setminus S; \Omega^{p}), 
\]
the relative cohomology of the pair $(X,X\setminus S)$ with coefficients in the sheaf $\Omega^{p}$ of holomorphic $p$-forms.

Together with integration theory, the relative \v{C}ech-Dolbeault cohomology has been used
to study the localization of characteristic classes, see \cite{suwa-ASPM, abate-bracci-suwa-tovena}, and has found more recent
applications to hyperfunction theory, see \cite{honda-izawa-suwa}.

Notice that if $X$ and $\tilde X$ are complex manifolds, $S$ and $\tilde S$ are closed sets
in $X$ and $\tilde X$ respectively and $f:\tilde X\to X$ is a holomorphic map such that $f(\tilde S)\subset S$ and
$f(\tilde X\setminus \tilde S)\subset f(X\setminus S)$, 
then $f$ induces a natural map in relative cohomology.
More precisely, let $U_0:= X\setminus S$, $\tilde U_0:= \tilde X\setminus \tilde S$
and let $U_1$, $\tilde U_1$ be open neighborhoods of $S$ and $\tilde S$
in $X$ and $\tilde X$ respectively, chosen in such a way that
$f(\tilde U_1)\subset U_1$. Let
$\mathcal{U}:=\left\lbrace U_0,U_1\right\rbrace$
and $\mathcal{\tilde U}:=\left\lbrace \tilde U_0,\tilde U_1\right\rbrace$
be open coverings of $X$ and $\tilde X$ respectively, then we have a morphism
$$
f^*:A^{\bullet,\bullet}(\mathcal{U},U_0)\lra 
A^{\bullet,\bullet}(\mathcal{\tilde U},\tilde U_0)
$$
defined on every element $(\xi_1,\xi_{01})\in A^{\bullet,\bullet}(\mathcal{U},U_0)$ as
$$
f^*(\xi_1,\xi_{01}):=(f^*\xi_1,f^*\xi_{01})
$$
which induces a morphism in relative cohomology
$$
f^*:H^{\bullet,\bullet}_{\bar D}(X,X\setminus S)\lra
H^{\bullet,\bullet}_{\bar D}(\tilde X,\tilde X\setminus \tilde S)\,.
$$

\subsection{Dolbeault-Thom morphism}\label{subsec:d-thom}
We consider a holomorphic vector bundle $\pi\colon E\to X$ of rank $k$ on a complex manifold $X$
and we identify $X$ with the image of the zero section. In this situation we have the Dolbeault-Thom class,
$\delbar$-Thom class for short, 
$\bar\Psi_{E}\in H^{k,k}_{\bar D}(E,E\setminus X)$ and the Dolbeault-Thom morphism, $\delbar$-Thom morphism for short, $\bar T_{E}\colon H^{p-k,q-k}_{\delbar}(X)\to H^{p,q}_{\bar D}(E,E\setminus X)$.

They are given as follows, see  \cite{abate-bracci-suwa-tovena, suwa-ASPM}. Consider the fibre product
\[
\SelectTips{cm}{}
\xymatrix
@C=.7cm
@R=.7cm
{\pi^{*}E\ar[r]\ar[d]_-{\varpi}& E\ar[d]^-{\pi}\\
 E\ar[r]^-{\pi} & X.}
\]
The bundle $\varpi:\pi^{*}E\to E$ admits the diagonal section $s_{\Delta}$, whose zero set is $X\subset E$.
The Dolbeault-Thom class $\bar\Psi_{E}$ is the localization of the top Atiyah class $a^{k}(\pi^{*}E)$ of $\pi^{*}E$ by $s_{\Delta}$. More precisely, let $W_{0}=E\setminus X$ and let $W_{1}$ be a neighbourhood of $X$ in $E$, and consider the
covering $\mathcal{W}=\{W_{0},W_{1}\}$ of $E$. For a $(1,0)$-connection $\nabla$ for $\pi^{*}E$, we denote by $a^{k}(\nabla)$ the $k$-th Atiyah form of $\nabla$, namely, $a^k(\nabla)=\big(\frac{\sqrt{-1}}{2\pi}\big)^{k}\sigma_k(K^{1,1})$, where $K^{1,1}$ is the $(1,1)$-component of the curvature seen as a $2$-form with values in $\mathrm{Hom}(E,E)$ and $\sigma_k$ denotes the $k$-th elementary symmetric polynomial, see \cite[Section 5]{suwa-ASPM} for more details. The class $a^{k}(\pi^{*}E)$ is represented in 
$H^{k,k}_{\delbar}(E)\simeq H^{k,k}_{\bar D}(\mathcal{W})$ by the triple 
$a^{k}(\nabla_{*})=(a^{k}(\nabla_{0}),a^{k}(\nabla_{1}),a^{k}(\nabla_{0},\nabla_{1}))$, where $\nabla_{i}$ is a $(1,0)$-connection for $\pi^{*}E$ on $W_{i}$, $i=0,1$, and  $a^{k}(\nabla_{0},\nabla_{1})$ is the  difference form of $\nabla_{0}$ and
$\nabla_{1}$. If we take $\nabla_{0}$ to be $s_{\Delta}$-trivial, we have the vanishing $a^{k}(\nabla_{0})=0$ and 
$a^{k}(\nabla_{*})$ defines a class in $H^{k,k}_{\bar D}(\mathcal{W},W_{0})=H^{k,k}_{\bar D}(E,E\setminus X)$ that is the {\em Dolbeault-Thom class} $\bar\Psi_{E}$ of $E$.

The {\em Dolbeault-Thom morphism}
\[
\bar T_{E}\colon H^{p-k,q-k}_{\delbar}(X)\lra H^{p,q}_{\bar D}(E,E\setminus X).
\]
is given by the cup product with $\bar\Psi_{E}$, {\itshape i.e.} if $\bar\Psi_{E}$ is represented by $(\psi_{1},\psi_{01})$, it is induced by
\[
\theta\mapsto (\psi_{1}\wedge\pi^{*}\theta,\psi_{01}\wedge\pi^{*}\theta).
\]
The inverse of $\bar T_{E}$ is given by the $\delbar$-integration along the fibres of $\pi$:
\[
\bar\pi_{*}:H^{p,q}_{\bar D}(E,E\setminus X)\lra H^{p-k,q-k}_{\delbar}(X).
\]
It is defined as follows. Let $T_{1}$ denote a bundle of discs of complex dimension $k$ in $W_{1}$ and
set $T_{01}=-\partial T_{1}$, which is a bundle of spheres of real dimension $2k-1$ endowed with
the orientation opposite to that of the boundary $\partial T_{1}$ of $T_{1}$. Set $\pi_{1}=\pi|_{T_{1}}$
and $\pi_{01}=\pi|_{T_{01}}$. Then we have the usual integration along the fibres
\[
(\pi_{1})_{*}:A^{r}(W_{1})\lra  A^{r-2k}(X)\quad\text{and}\quad (\pi_{01})_{*}:A^{r-1}(W_{01})\lra A^{r-2k}(X).
\]
The map $(\pi_{1})_{*}$ sends a $(p,q)$-form to a $(p-k,q-k)$-form, while, if $\xi_{01}$ is a $(p,q-1)$-form
on $W_{01}$,
$(\pi_{01})_{*}(\xi_{01})$ consists of  $(p-k,q-k)$ and  $(p-k+1,q-k-1)$-components. We define
\[
(\bar\pi_{01})_{*}:A^{p,q-1}(W_{01})\lra A^{p-k,q-k}(X)
\]
by taking the $(p-k,q-k)$-component of $(\pi_{01})_{*}(\xi_{01})$, then
\[
\bar\pi_{*}\xi=(\pi_{1})_{*}\xi_{1}+(\bar\pi_{01})_{*}\xi_{01}.
\]
In this situation,
\[
\bar\pi_{*}\circ \bar T_{E}=1.
\]
Thus $\bar\pi_{*}$ is surjective and $\bar T_{E}$ gives a splitting of
\[
0\lra\text{ker}\bar\pi_{*}\lra H^{p,q}_{\bar D}(E,E\setminus X)\overset{\bar\pi_{*}}\lra H^{p-k,q-k}_{\delbar}(X)\lra 0.
\]
For the $\delbar$-Thom class $\bar\Psi_{E}\in H^{k,k}_{\bar D}(E,E\setminus X)$, we have
$\bar\pi_{*}\bar \Psi_{E}=[1]\in H^{0,0}_{\delbar}(X)$.

\subsection{$\del\delbar$-Lemma and Hodge structures}\label{ss-deldelbar}
Although these may be well-known to experts, we recall what 
the $\del\delbar$-Lemma means and some alternative ways of saying that for later use.

Let $X$ be a complex manifold.  The de~Rham complex $(A^{\bullet}(X),d)$ of $X$ is the single complex associated with the double complex $(A^{\bullet,\bullet}(X),\partial,\delbar)$, $d=\partial+\delbar$.
Recall that \cite{deligne-griffiths-morgan-sullivan} $X$ satisfies the $\partial\delbar$-Lemma if 
\begin{equation}\label{conddeldelbar}
\op{ker}\partial\cap\op{ker}\delbar\cap\op{im}d=
\op{im}\partial\delbar.
\end{equation}

We describe the above property in terms of filtrations.
Note that $A^{\bullet}(X)$ 
 has two natural filtrations.
The first  filtration on $A^{h}(X)$ is given by
\[
{}'\hspace{-.5mm}F^{p}A^{h}(X)=\bigoplus_{i= p}^{h} A^{i,h-i}(X).
\]
It induces a  filtration on $H^{h}_{dR}(X)$ by
\[
{}'\hspace{-.5mm}F^{p}H^{h}_{dR}(X)=\op{ker} d^{h}\cap {}'\hspace{-.5mm}F^pA^{h}(X)/\op{im}d^{h-1}\cap {}'\hspace{-.5mm}F^pA^{h}(X).
\]

The second filtration on $A^{h}(X)$ is given by
\[
{}''\hspace{-.5mm}F^{q}A^{h}(X)=\bigoplus_{j=q}^{h} A^{h-j,j}(X)
\]
and it induces a filtration $({}''\hspace{-.5mm}F^{q}H^{h}_{dR}(X))$ on $H^{h}_{dR}(X)$.

Since $\overline{A^{q,p}(X)}=A^{p,q}(X)$, we may identify the filtration $(\overline{{}'\hspace{-.5mm}F^{q}A^{h}(X)})$ conjugate to $({}'\hspace{-.5mm}F^{q}A^{h}(X))$ with the second 
filtration:
$\overline{{}'\hspace{-.5mm}F^{q}A^{h}(X)}={}''\hspace{-.5mm}F^{q}A^{h}(X),
$
which leads to the identification
\[
\overline{{}'\hspace{-.5mm}F^{q}H^{h}_{dR}(X)}={}''\hspace{-.5mm}F^{q}H^{h}_{dR}(X).
\]

We say that the filtration $({}'\hspace{-.5mm}F^pH^{h}_{dR}(X))$ is
a {\em Hodge filtration} of weight $h$ if
\[
H^{h}_{dR}(X)=\bigoplus_{p+q=h} {}'\hspace{-.5mm}F^pH^{h}_{dR}(X)\cap\overline{{}'\hspace{-.5mm}F^qH^{h}_{dR}(X)}.
\]

\begin{lemma}\label{lemmahf} The filtration $({}'\hspace{-.5mm}F^pH^{h}_{dR}(X))$ is a Hodge filtration of weight $h$ if and only if
\[
H^{h}_{dR}(X)={}'\hspace{-.5mm}F^pH^{h}_{dR}(X)\oplus\overline{{}'\hspace{-.5mm}F^qH^{h}_{dR}(X)}\qquad\text{for every $(p,q)$ with $p+q=h+1$}.
\]

Moreover, if this is the case, there is a canonical  \iso
\[
{}'\hspace{-.5mm}F^pH^{h}_{dR}(X)\cap\overline{{}'\hspace{-.5mm}F^qH^{h}_{dR}(X)}\simeq {}'\hspace{-.5mm}G^{p}H^{h}_{dR}(X)\qquad\text{for every $(p,q)$ with $p+q=h$},
\]
where ${}'\hspace{-.5mm}G^{p}H^{h}_{dR}(X)={}'\hspace{-.5mm}F^{p}H^{h}_{dR}(X)/{}'\hspace{-.5mm}F^{p+1}H^{h}_{dR}(X)$.
\end{lemma}
\begin{proof} It is rather straightforward to show the equivalence of two expressions for Hodge filtrations.
We only indicate a proof of the last statement for later use.
In the sequel we denote $H^{h}_{dR}(X)$ by $H^{h}$.

For $c\in {}'\hspace{-.5mm}F^pH^{h}$ we denote by $[c]^{p}$ its class
in ${}'\hspace{-.5mm}G^{p}H^{h}$. We define a morphism
\[
{}'\hspace{-.5mm}F^pH^{h}\cap\overline{{}'\hspace{-.5mm}F^qH^{h}}\lra {}'\hspace{-.5mm}G^{p}H^{h}\qquad\text{by}\ \ c\mapsto [c]^{p}
\]
and show that it is an \iso. For the surjectivity, take $[c]^{p}\in {}'\hspace{-.5mm}G^{p}H^{h}$, $c\in {}'\hspace{-.5mm}F^{p}H^{h}$.
Then we may write uniquely  $c=\sum_{i=0}^{h}c^{i,h-i}$ with $c^{i,h-i}\in {}'\hspace{-.5mm}F^{i}H^{h}\cap\overline{{}'\hspace{-.5mm}F^{h-i}H^{h}}$.
We have $[c]^{p}=[c']^{p}$, $c'=\sum_{i=0}^{p}c^{i,h-i}$. Since $\sum_{i=p+1}^{h}c^{i,h-i}\in {}'\hspace{-.5mm}F^{p+1}H^{h}\subset {}'\hspace{-.5mm}F^{p}H^{h}$, we have $c'\in {}'\hspace{-.5mm}F^{p}H^{h}$. On the other hand, $c'$ is also in $\overline{{}'\hspace{-.5mm}F^{q}H^{h}}$ and $c'\mapsto [c]^{p}$.
For the injectivity, take $c\in {}'\hspace{-.5mm}F^{p}H^{h}\cap\overline{{}'\hspace{-.5mm}F^{q}H^{h}}$ \st\ $[c]^{p}=0$.
This means that $c\in {}'\hspace{-.5mm}F^{p+1}H^{h}\cap\overline{{}'\hspace{-.5mm}F^{q}H^{h}}=0$.
\end{proof}

The spectral sequence associated with the first filtration of $A^{\bullet}(X)$ is  the
Fr\"olicher spectral sequence \cite{frolicher}, for which we have 
\[
E^{p,q}_{1}\simeq H^{p,q}_{\delbar}(X),\qquad
E^{p,q}_{\infty}\simeq {}'\hspace{-.5mm}G^{p}H^{p+q}_{dR}(X),
\]


\begin{proposition}[\cite{deligne-griffiths-morgan-sullivan}] A complex manifold $X$ satisfies the $\partial\delbar$-Lemma if and only if  the following two
conditions hold:
\begin{enumerate}
\item[\rm (1)] the Fr\"olicher spectral sequence degenerates at $E_{1}$,
\item[\rm (2)] the filtration $({}'\hspace{-.5mm}F^pH^{h}_{dR}(X))$ is a Hodge filtration of weight $h$ for
every $h\ge 0$.
\end{enumerate}
\end{proposition}

Note that every element of ${}'\hspace{-.5mm}G^{p}H^{p+q}_{dR}(X)$
 is expressed as $[[\omega]]^{p}$, where $\omega$ is a $d$-closed form in ${}'\hspace{-.5mm}F^{p}A^{p+q}(X)$,
$[\omega]$ is the class of $\omega$ in 
${}'\hspace{-.5mm}F^{p}H^{p+q}_{dR}(X)$ and $[[\omega]]^{p}$ is the class of $[\omega]$ in ${}'\hspace{-.5mm}G^{p}H^{p+q}_{dR}(X)$. The condition $d\omega=0$ implies that $\delbar\omega^{p,q}=0$ when we write $\omega=\sum_{i=p}^{p+q}\omega^{i,p+q-i}$.

The condition (1) above is equivalent to saying that, for every $(p,q)$, the assignment $[[\omega]]^{p}\mapsto [\omega^{p,q}]$ is well-defined and induces an \iso
\[
{}'\hspace{-.5mm}G^{p}H^{p+q}_{dR}(X) \overset\sim\lra H^{p,q}_{\delbar}(X).
\]

Recall that $A^{p,q}(X)=\overline{A^{q,p}(X)}$ and
$A^{h}(X)=\bigoplus_{p+q=h} A^{p,q}(X)$. We ask when these relations carry on to the cohomologies.

\begin{definition}\label{defH} 1. We say that $X$ admits a Hodge \str\ of weight $h$,
if there exist \iso s
\[
H^{p,q}_{\delbar}(X)\simeq \overline{H^{q,p}_{\delbar}(X)},\  p+q=h,\quad\text{and}\quad H^{h}_{dR}(X)\simeq\bigoplus_{p+q=h}H^{p,q}_{\delbar}(X).
\]

\noindent
2. A Hodge \str\ as above is said to be {\em natural}, if the following conditions hold:
\begin{enumerate}
\setlength{\leftskip}{2.3mm}
\item[(H1)] Every class in $H^{p,q}_{\delbar}(X)$, $p+q=h$, admits a representative $\omega$ with
$\partial\omega=0$ and $\delbar\omega=0$, i.e., $d\omega=0$. Moreover, the assignment $\omega\mapsto \bar\omega$ induces the first \iso\ above.
\item[(H2)] Every class in $H^{h}_{dR}(X)$ admits a representative $\omega$ which may be written
$\omega=\sum_{p+q=h}\omega^{p,q}$, where $\omega^{p,q}$ is a  $(p,q)$-form with $d\omega^{p,q}=0$. Moreover, the assignment $\omega\mapsto(\omega^{p,q})_{p+q=h}$ induces the second \iso\ above.
\end{enumerate}
\end{definition}

\begin{remark}  In (H1) above, $\overline{H^{q,p}_{\delbar}(X)}$ denotes the vector space conjugate to 
$H^{q,p}_{\delbar}(X)$, i.e., the vector space with underlying set  $H^{q,p}_{\delbar}(X)$ and the 
complex multiplication given by $c\cdot \omega=\bar c\, \omega$. We may rephrase (H1) as:
\begin{enumerate}
\setlength{\leftskip}{3mm}
\item[(H1)$'$] Every class in $H^{p,q}_{\delbar}(X)$, $p+q=h$, admits a representative $\omega$ with
$\partial\omega=0$ and $\delbar\omega=0$, i.e., $d\omega=0$. Moreover, the assignment $\omega\mapsto \omega$ induces
an \iso\
$H^{p,q}_{\delbar}(X)\simeq 
H^{p,q}_{\partial}(X)$.
\end{enumerate}
\end{remark}

\begin{remark}
See \cite[Proposition 4.3]{couv} for an example of a compact complex manifold with a non-natural Hodge structure.
\end{remark}

\begin{proposition}\label{ch9thHodgedec} A complex manifold $X$ admits a natural Hodge \str\
of weight $h$ if and only if 
the following conditions hold\,{\rm :}
\begin{enumerate}
\item[{\rm (i)}] the morphism $\op{ker}d\cap {}'\hspace{-.5mm}F^{p}A^{h}(X)\ra A^{p,h-p}(X)$,
$\omega\mapsto\omega^{p,h-p}$, induces an \iso\ ${}'\hspace{-.5mm}G^{p}H^{h}_{dR}(X)\simeq H^{p,h-p}_{\delbar}(X)$ for every $p$,
\item [{\rm (ii)}] $({}'\hspace{-.5mm}F^{p}H^{h}_{dR}(X))$ is a Hodge filtration on $H^{h}_{dR}(X)$ of   weight $h$.
\end{enumerate}
\end{proposition}

\begin{proof} Suppose $X$ admits the natural Hodge \str\
of weight $h$. We claim that there is an \iso
\begin{equation}\label{ch9HF}
{}'\hspace{-.5mm}F^{p}H^{h}_{dR}(X)\simeq\bigoplus_{i=p}^{h}H^{i,h-i}_{\delbar}(X)
\end{equation}
compatible with the one in (H2) in the sense that the following is commutative:
\[
\SelectTips{cm}{}
\xymatrix
@C=.25cm
@R=.3cm
{{}'\hspace{-.5mm}F^{p}H^{h}_{dR}(X)
\ar@{-}[r]^-{\sim}
\ar@{}[d]|{\bigcap} &\bigoplus_{i=p}^{h}H^{i,h-i}_{\delbar}(X)\ar@{}[d]|{\bigcap}\\
H^{h}_{dR}(X)\ar@{-}[r]^-{\sim} &\bigoplus_{i=0}^{h}H^{i,h-i}_{\delbar}(X).}
\]
For this, 
take $\theta\in\op{ker}d^{h}\cap {}'\hspace{-.5mm}F^{p}A^{h}(X)$ and write 
$\theta=\sum_{i=p}^{h}\theta^{i,h-i}$ with $\theta^{i,h-i}\in A^{i,h-i}(X)$. From $d\theta=0$ and (H1), we see that there
exist $\omega^{i,h-i}$ and $\alpha^{i,h-i}$ in $A^{i,h-i}(X)$, $p\le i\le h$,  \st\ $d\omega^{i,h-i}=0$ and that 
\[
\theta^{i,h-i}=\omega^{i,h-i}+\partial\alpha^{i-1,h-i}+\delbar\alpha^{i,h-i-1},
\]
where we set $\alpha^{p-1,h-p}=0$.
Then we have
\[
\theta=\sum_{i=p}^{h}\omega^{i,h-i}+d\sum_{i=p}^{h}\alpha^{i,h-i-1}.
\]
By (H2), the assignment $\theta\mapsto (\omega^{i,h-i})_{i=p}^{h}$ induces a well-defined morphism
${}'\hspace{-.5mm}F^{p}H^{h}_{dR}(X)\ra \bigoplus_{i=p}^{h}H^{i,h-i}_{\delbar}(X)$ compatible with
the \iso\ of (H2). It is obviously injective. The surjectivity follows from (H1) and
it is the desired \iso.

From \eqref{ch9HF}, we have ${}'\hspace{-.5mm}G^{p}H^{h}_{dR}(X)\simeq H^{p,h-p}_{\delbar}(X)$
and the correspondence is the one  as given in (i).

We also have
an \iso\ $\overline{{}'\hspace{-.5mm}F^{q}H^{h}_{dR}(X)}\simeq\bigoplus_{j=q}^{h}H^{h-j,j}_{\delbar}(X)$ compatible with the one in (H2). Thus for $(p,q)$ with $p+q=h+1$, $H^{h}_{dR}(X)=
{}'\hspace{-.5mm}F^{p}H^{h}_{dR}(X)\oplus \overline{{}'\hspace{-.5mm}F^{q}H^{h}_{dR}(X)}$
and we have (ii).

Now we prove the converse. 
The condition (ii) implies
 (cf. Lemma \ref{lemmahf})
\begin{equation}\label{ch9H2}
H^{h}_{dR}(X)=\bigoplus_{p+q=r}{}'\hspace{-.5mm}F^{p}H^{h}_{dR}(X)\cap\overline{{}'\hspace{-.5mm}F^{q}H^{h}_{dR}(X)}\qquad\text{and}
\end{equation}
\begin{equation}\label{ch9H3}
{}'\hspace{-.5mm}F^{p}H^{h}_{dR}(X)\cap\overline{{}'\hspace{-.5mm}F^{q}H^{h}_{dR}(X)}\simeq {}'\hspace{-.5mm}G^{p}H^{h}_{dR}(X),\quad h=p+q.
\end{equation}
From the condition (i) and  \eqref{ch9H3}, we have 
\[
\begin{aligned}
H^{p,q}_{\delbar}(X)&\simeq {}'\hspace{-.5mm}G^{p}H^{h}_{dR}(X)\simeq
{}'\hspace{-.5mm}F^{p}H^{h}_{dR}(X)\cap \overline{{}'\hspace{-.5mm}F^{q}H^{h}_{dR}(X)}\\
&\simeq \overline{{}'\hspace{-.5mm}G^{q}H^{h}_{dR}(X)}\simeq \overline{H^{q,p}_{\delbar}(X)},
\quad h=p+q.
\end{aligned}
\]
We look at the correspondence above.
Take $c\in
{}'\hspace{-.5mm}F^{p}H^{h}_{dR}(X)\cap \overline{{}'\hspace{-.5mm}F^{q}H^{h}_{dR}(X)}$, we may write
$c=[\omega_{1}]=[\omega_{2}]$, where $\omega_{1}=\sum_{i=p}^{h}\omega_{1}^{i,h-i}\in {}'\hspace{-.5mm}F^{p}A^{h}(X)$, $d\omega_{1}=0$ (thus $\delbar\omega_{1}^{p,q}=0$) and 
$\omega_{2}=\sum_{j=q}^{h}\omega_{1}^{h-j,j}\in \overline{{}'\hspace{-.5mm}F^{q}A^{h}(X)}$, $d\omega_{2}=0$ (thus $\partial\omega_{2}^{p,q}=0$). Then the correspondence is given by $[\omega_{1}^{p,q}]
\leftrightarrow [\overline{\omega_{2}^{p,q}}]$. From $[\omega_{1}]=[\omega_{2}]$, we see that there
exist $\theta^{p,q-1}\in A^{p,q-1}(X)$ and $\theta^{p-1,q}\in A^{p-1,q}(X)$ \st\
\[
\omega_{1}^{p,q}-\omega_{2}^{p,q}=\partial\theta^{p-1,q}+\delbar\theta^{p,q-1}.
\]
Then $\omega=\omega^{p,q}_{1}-\delbar\theta^{p,q-1}=\omega^{p,q}_{2}+\partial\theta^{p-1,q}$ is a representative
as in (H1).
From \eqref{ch9H2}, \eqref{ch9H3} and the condition (i), we have (H2).
\end{proof}

\begin{corollary}\label{cordeldelbar} A complex manifold satisfies the $\partial\delbar$-Lemma if and only if it admits a natural
Hodge \str\ of weight $h$ for every $h$.
\end{corollary}

\begin{remark} If we use the  {\em Bott-Chern cohomology} $H^{\bullet,\bullet}_{BC}(X) = \frac{\ker\del \cap \ker\delbar}{\mathrm{im}\,\del\delbar}$, the condition \eqref{conddeldelbar}
means that
the morphism
$$ H^{\bullet,\bullet}_{BC}(X) \lra H^{\bullet}_{dR}(X)$$
induced by the identity is injective.
A numerical characterization of the $\del\delbar$-Lemma in terms of the dimension of the Bott-Chern cohomology and the Betti numbers is provided in \cite{angella-tomassini-3} and in \cite{angella-tardini} using only Bott-Chern numbers.
\end{remark}

\section{Dolbeault cohomology of the projectivization of a holomorphic vector bundle}\label{secproj}

Let $X$ be a smooth manifold.
Also let  $\pi:V\to X$ be a complex vector bundle of rank $k$ and  denote by $\rho:\mathbb{P}(V)\ra X$ its
projectivization. We may regard $H^{\bullet}_{dR}(\mathbb{P}(V))$ as an $H^{\bullet}_{dR}(X)$-module
(in fact, $H^{\bullet}_{dR}(X)$-algebra).  Here we regard it as a right module by our convention and the
module structure is given 
by $c\cdot a=c\smallsmile \rho^{*}(a)$ for $c\in H^{\bullet}_{dR}(\mathbb{P}(V))$ and $a\in H^{\bullet}_{dR}(X)$, where $\smallsmile$ denotes the cup product. In the sequel  it will be simply denoted by $\cdot$, if there is no fear of confusion.

In the above situation we have the tautological bundle $T$ on $\mathbb{P}(V)$, which is a rank one subbundle of 
$\rho^{*}V$ with the universal  bundle $Q$ as the quotient so that we have an exact sequence of vector bundles on $\mathbb{P}(V)$:
\begin{equation}\label{exactuniv0}
0\lra T\lra \rho^{*}V\lra  Q\lra 0.
\end{equation}
Recall that 
$\rho^*V=\{\,(v,l)\in V\times \mathbb{P}(V)\mid \pi(v)=\rho(l)\,\}$. We may think of a point $l$ in $\mathbb P(V)$ as a line in $V_{x}=\C^{k}$, $x=\pi(v)$,  and we have
$T=\{\,(v,l)\in \rho^{*}V\mid v\in l\,\}$.

We  recall the following, which is a direct consequence of the Leray-Hirsch theorem (cf. \cite[Proposition page 606]{griffiths-harris}, \cite[Lemma 7.32]{voisin-1}):
\begin{proposition}\label{prop:deRham-projectivized}
In the above situation, $H^{\bullet}_{dR}(\mathbb{P}(V))$ is a free $H^{\bullet}_{dR}(X)$-module with
basis $1, \gamma,\dots, \gamma^{k-1}$, where $\gamma=c^{1}(T)$ is the first Chern class of $T$.
\end{proposition}

The essential point in the above is that the restriction of $\gamma$ to each fibre, which is the projective space $\mathbb{P}^{k-1}$, is the first Chern of the tautological bundle (dual of the hyperplane bundle) on $\mathbb{P}^{k-1}$
and that their powers up to the $(k-1)$-st form a $\mathbb{C}$-basis of $H^{\bullet}_{dR}(\mathbb{P}^{k-1})$. As an $H^{\bullet}_{dR}(X)$-algebra, $H^{\bullet}_{dR}(\mathbb{P}(V))$ is generated by $\gamma$ with the single relation
\begin{equation}\label{relchern}
\sum_{i=0}^{k}(-1)^{i}\gamma^{i}\cdot\rho^{*}c^{k-i}(V)=0,
\end{equation}
where $c^{k-i}(V)$ is the $(k-i)$-th Chern class of $V$. The relation can be seen from $c(T)\cdot c(Q)=\rho^{*}c(V)$, the relation among the total Chern classes, which follows from \eqref{exactuniv0}.

 If we take a metric connection for 
 $T$, its curvature form $\kappa$
is of type $(1,1)$ and is simultaneously $d$- and $\delbar$-closed. We also have $\bar\kappa=-\kappa$. The class of $\frac {\sqrt{-1}}{2\pi}\kappa$
in $H^{2}_{dR}(\mathbb{P}(V))$ is the first Chern class $\gamma=c^{1}(T)$ and its class in 
$H^{1,1}_{\delbar}(\mathbb{P}(V))$ is the first Atiyah class $a^{1}(T)$.
Note that they cannot be compared
directly on the cohomology level, in general. However, their restrictions to each fibre of $\mathbb{P}(V)\ra X$
may be identified, as the fibre is $\mathbb{P}^{k-1}$ and it satisfies the $\del\delbar$-Lemma.

\begin{proposition}\label{prop:dolbeault-projectivized}
Let $V\ra X$ be a holomorphic vector bundle of rank $k$ on a compact complex manifold $X$. Then $H^{\bullet,\bullet}_{\delbar}(\mathbb{P}(V))$ is a free $H^{\bullet,\bullet}_{\delbar}(X)$-module with
basis $1, \alpha,\dots, \alpha^{k-1}$, where $\alpha=a^{1}(T)$ is the first Atiyah class of $T$.
\end{proposition}

\begin{proof} By \cite[Lemma 18]{cordero-fernandez-gray-ugarte}, we see that $H^{\bullet,\bullet}_{\delbar}(\mathbb{P}(V))$ is generated by $\alpha$ as an $H^{\bullet,\bullet}_{\delbar}(X)$-algebra.
We have a relation as \eqref{relchern}, replacing $\gamma$ and $c^{k-i}(V)$ with $\alpha$ and $a^{k-i}(V)$,
the $(k-i)$-th Atiyah class of $V$, from which we see that $1,\alpha,\dots,\alpha^{k-1}$ generate $H^{\bullet,\bullet}_{\delbar}(\mathbb{P}(V))$  as an $H^{\bullet,\bullet}_{\delbar}(X)$-module.
The proposition follows from the following:

\noindent{\itshape Claim.} $1,\alpha,\dots,\alpha^{k-1}$ are linearly independent over $H^{\bullet,\bullet}_{\delbar}(X)$. 

To prove this, we look at the $\mathbb{C}$-algebra structure of $H^{\bullet,\bullet}_{\delbar}(\mathbb{P}(V))$. Let $n=\dim X$ and $h^{p,q}=\dim H^{p,q}_{\delbar}(X)$. For each $(p,q)$ with $h^{p,q}\ne 0$, we take a basis
$\{u^{p,q}_{i}\}_{1\le i\le h}$  of $H^{p,q}_{\delbar}(X)$ so that  $\{u^{n-p,n-q}_{i'}\}_{1\le i'\le h}$ is the basis of $H^{n-p,n-q}_{\delbar}(X)$ dual to $\{u^{p,q}_{i}\}$ via the Kodaira-Serre duality, $h=h^{p,q}=h^{n-p,n-q}$:
\[
\int_{X}u^{p,q}_{i}\cdot u^{n-p,n-q}_{i'}=\pm\delta_{ii'}.
\]

Obviously $\{\alpha^{r}\cdot\rho^{*}u^{p,q}_{i}\}_{0\le r\le k-1,\, p,\, q,\, i}$
span the $\mathbb{C}$-vector space $H^{\bullet,\bullet}_{\delbar}(\mathbb{P}(V))$.
We show that they are linearly independent
over $\mathbb{C}$, which will prove the claim and the proposition.
For this we introduce a
relation $>$ in the set $\Lambda$ of indices $\lambda=(r,p,q,i)$ by saying that $(r_{1},p_{1},q_{1},i_{1})>(r_{2},p_{2},q_{2},i_{2})$ if one of the following holds:
\begin{enumerate}
\item $2r_{1}+p_{1}+q_{1}>2r_{2}+p_{2}+q_{2}$,
\item $2r_{1}+p_{1}+q_{1}=2r_{2}+p_{2}+q_{2}$ and $p_{1}+q_{1}>p_{2}+q_{2}$,
\item $r_{1}=r_{2}$, $p_{1}+q_{1}=p_{2}+q_{2}$ and $p_{1}>p_{2}$,
\item $r_{1}=r_{2}$, $p_{1}=p_{2}$, $q_{1}=q_{2}$ and $i_{1}>i_{2}$.
\end{enumerate}
With this, $\Lambda$ becomes a totally ordered set. Let $\Lambda'$ denote the set $\Lambda$ with
the order defined by reversing the inequalities in (1), (2) and (3) and keeping that in (4) above.
We consider the matrix 
$(v_{\lambda}\cdot v_{\lambda'})_{(\lambda,\lambda')\in\Lambda\times\Lambda'}$, where, for 
$\lambda=(r,p,q,i)$, $v_{\lambda}=\alpha^{r}\cdot\rho^{*}u^{p,q}_{i}$ and similarly for $v_{\lambda'}$.
On the diagonal, we have $v_{\lambda}\cdot v_{\lambda'}$ for which $r+r'=k-1$, $p+p'=n$, $q+q'=n$ and $i=i'$ (note that this makes sense as $p+p'=n$ and $q+q'=n$),  when we write $\lambda=(r,p,q,i)$ and $\lambda'=(r',p',q',i')$. In this case, noting that $\rho_{*}\alpha^{k-1}=(-1)^{k-1}$, as the restriction of $\alpha$ to each fibre is the first Atiyah class of the dual of the hyperplane bundle, 
by the projection formula,
\[
\int_{\mathbb{P}(V)}(\alpha^{r}\cdot \rho^{*}u^{p,q}_{i})\cdot(\alpha^{r'}\cdot\rho^{*}u^{p',q'}_{i'})=
\rho_{*}\alpha^{k-1}\cdot \int_{X}u^{p,q}_{i}\cdot u^{n-p,n-q}_{i}=\pm 1.
\]

On the upper triangle, but off the diagonal, we have $v_{\lambda}\cdot v_{\lambda'}$ for which one of the following holds:
\begin{enumerate}
\item $2r+2r'+p+p'+q+q'>2(n+k-1)$,
\item $2r+2r'+p+p'+q+q'=2(n+k-1)$ and $p+p'+q+q'>2n$,
\item $r+r'=k-1$, $p+p'+q+q'=2n$ and $p+p'>n$,
\item $r+r'=k-1$, $p+p'=n$, $q+q'=n$ and $i<i'$.
\end{enumerate}
Recalling that $\dim X=n$ and $\dim \mathbb{P}(V)=n+k-1$, we have, in the case (1), (2) or (3),
\[
(\alpha^{r}\cdot\rho^{*}u^{p,q}_{i})\cdot(\alpha^{r'}\cdot\rho^{*}u^{p',q'}_{i'})=0,
\]
by dimension reason. In the case (4), $\int_{\mathbb{P}(V)}v_{\lambda}\cdot v_{\lambda'}=0$ by a similar
computation as above.
Thus the Kodaira-Serre dual
of the matrix $(v_{\lambda}\cdot v_{\lambda'})$ is triangular with $\pm 1$'s along the diagonal,
which shows that the $\alpha^{r}\cdot\rho^{*}u^{p,q}_{i}$'s are linearly independent over $\mathbb{C}$.
\end{proof}

\begin{corollary}\label{cor:deldelbar-E}
Let $X$ be a compact complex manifold and
$V\to X$  a holomorphic fibre bundle on $X$. If $X$ satisfies the $\del\delbar$-lemma, so does $\mathbb{P}(V)$.
\end{corollary}
\begin{proof}
The statement follows from Corollary \ref{cordeldelbar} and Propositions \ref{prop:deRham-projectivized} and 
\ref{prop:dolbeault-projectivized}, noting that $\gamma$ and $\alpha$ are both represented by the same 
form $\frac{\sqrt{-1}}{2\pi}\kappa$ as above.
\end{proof}

\section{Hodge structures under blow-ups}

We can now prove explcit expressions for the de Rham (Proposition \ref{isodR}) and Dolbeault (Proposition \ref{isoDol}) cohomologies of the blow-up and then Theorem \ref{thm:blowup}. Compare also \cite[Theorem 1.3]{yang-yang} for similar results using Bott-Chern cohomology, and \cite[Corollary 25]{stelzig-doublecomplex} for a clear statement and argument.

Let $X$ be a compact complex manifold of  dimension $n$ and $Z$  a closed complex submanifold of codimension $k$. Also let   $\tau \colon \tilde X:=\tilde X_Z \to X$ be the blow-up of $X$ along  $Z$ with exceptional divisor $\exc=\mathbb{P}(N_{Z|X})$.
Here we assume that 
\begin{equation}\label{htnbd}
\text{$Z$ admits a holomorphically contractible neighbourhood}
\end{equation}
 that is, there exists $U\supset Z$ with $r \colon U \to Z$ holomorphic and $r|_Z=\mathrm{id}$. In this case $\exc$ also admits a holomorphically contractible neighbourhood $\tilde U\supset \exc$ with $\tilde r \colon \tilde U \to \exc$ holomorphic and $\tilde r|_\exc=\mathrm{id}$.
Thus we have the following diagram:
\begin{equation}\label{diagfund}
\SelectTips{cm}{}
\xymatrix
@R=.7cm
 {H^{p-k,q-k}_{\delbar}(Z)\ar[d]_{\bar\chi}&H^{p,q}_{\bar D}(X,X\setminus Z)\ar[d]^{\tau^{*}}\ar[l]_{\bar r_{*}}\\
H^{p-1,q-1}_{\delbar}(\exc) &H^{p,q}_{\bar D}(\tilde X,\tilde X\setminus \exc)\ar[l]_{\bar{\tilde r}_{*}},}
\end{equation}
where the horizontal arrows are the $\delbar$-integrations along the fibres, $\tau^{*}$ is the morphism
induced by $\tau$ and $\bar\chi$ is given by $z\mapsto a^{k-1}(Q)\cdot \tau_{\exc}^{*}z$, see the proof below for details. Here we spend some words to clarify the heavy notation: accordingly with \cite{suwa-ASPM}, the bar refers to the holomorphic aspects of the theory, while the tilde concerns to the level of the blow-up.

We do not know whether or not the diagram \eqref{diagfund} is commutative. The first condition in \eqref{eq:technical} below
is apparently weaker than the commutativity
(cf. Remark \ref{rmk:assumptions-special-cases}.\,(5) below).

\begin{theorem}\label{thm:blowup} Let $X$ be a compact complex manifold and $Z$  a closed  submanifold of $X$. Also let $\tau \colon \tilde X_Z \to X$ be the blow-up of $X$ along  $Z$. Assume that the conditions \eqref{htnbd} above and
\begin{equation}\label{eq:technical}
\op{im}\bar{\tilde r}_{*}\circ\tau^{*}\subset\op{im}\bar\chi,\qquad\op{ker}\bar{\tilde r}_{*}\subset\op{im}\tau^{*}
\end{equation}
hold. 
Then, if both $X$ and $Z$ admit a Hodge structure, so does $\tilde X_Z$.
\end{theorem}

\begin{proof}
\noindent{\itshape Algebraic preliminaries.}
We quote the following lemma, see for instance \cite[Lemme II.6]{blanchard}:

\begin{lemma}\label{lemalg} Let $R$ be a commutative ring with unity and let
\[
\SelectTips{cm}{}
\xymatrix
{A_{1}\ar[r]\ar@{->>}[d]^-{f_{1}}& A_{2}\ar[r]\ar@{>->}[d]^-{f_{2}}& A_{3}\ar[r]\ar[d]^-{f_{3}}& A_{4}\ar[r]\ar[d]^-{f_{4}}_{\wr}& A_{5}\ar@{>->}[d]^-{f_{5}}\\
 B_{1}\ar[r] & B_{2}\ar[r]^-{g}& B_{3}\ar[r]& B_{4}\ar[r]& B_{5}}
\]
be a commutative diagram of $R$-modules with exact rows \st\ 
$f_{1}$ is surjective, $f_{2}$ and $f_{5}$ are injective and $f_{4}$ is an \iso.
Then $f_{3}$ is injective and $g$ induces an \iso
\[
\tilde g:B_{2}/f_{2}A_{2}\overset\sim\lra B_{3}/f_{3}A_{3}.
\]
\end{lemma}

In the above situation, we have  the diagram with an exact row:
\[
\SelectTips{cm}{}
\xymatrix
@C=.7cm
@R=.6cm
{0\ar[r]&A_{3}\ar[r]^{f_{3}}&B_{3}\ar[r]^-{\pi}& B_{3}/f_{3}A_{3}\ar[r]&0\\
{} &{}&{} & B_{2}/f_{2}A_{2}\ar[u]^-{\wr}_{\tilde g}\ar@{.>}[ul]^{\eta},}
\]
where $\pi$ is the canonical surjection.
If there is a splitting $\eta:B_{2}/f_{2}A_{2}\ra B_{3}$,  i.e., a morphism with $(\tilde g)^{-1}\circ\pi\circ\eta=\op{id}$, we have an \iso
\[
A_{3}\oplus (B_{2}/f_{2}A_{2})\overset\sim\lra B_{3},\qquad (a,[b])\mapsto f_{3}(a)+\eta ([b]).
\]
Note that the \iso\ depends on the splitting.

In the sequel, we try to express the cohomology of $\tilde X$ in terms of those of $X$ and $Z$ using the above.

\medskip
\noindent{\itshape de Rham cohomology.}
Let us start with the de Rham case.
Note that, for this case, the assumption \eqref{htnbd}
 (or \eqref{eq:technical}) is not necessary; for the map $r$, simply take the one given by the Tubular Neighbourhood Theorem, although it is only smooth that is sufficient.

Considering the  exact sequence \eqref{lexcdR} for the pairs $(X,X \setminus Z)$ and $(\tilde X, \tilde X \setminus \exc)$,
we have the commutative diagram with exact rows:
\[
{\resizebox{\textwidth}{!}{
\SelectTips{cm}{}
\xymatrix
{H^{h-1}_{dR}(X\setminus Z)\ar[r]^-{\delta}\ar[d]^-{\tau^*}_-{\wr} & H^{h}_{D}(X,X\setminus Z)\ar[r] ^-{j^*}\ar[d]^-{\tau^*} &H^h_{dR}(X)\ar[r]^-{i^{*}}\ar[d]^-{\tau^*} & H^h_{dR}(X\setminus Z)\ar[r]^-{\delta}\ar[d]^-{\tau^*}_-{\wr} & H^{h+1}_{D}(X,X\setminus Z)\ar[d]^-{\tau^*} \\
H^{h-1}_{dR}(\tilde X \setminus \exc)\ar[r]^-{\delta} & H^{h}_{D}(\tilde X,\tilde X \setminus \exc)\ar[r]^-{j^*} & H^h_{dR}(\tilde X)\ar[r]^-{i^{*}}& H^h_{dR}(\tilde X \setminus \exc)\ar[r]^-{\delta} & H^{h+1}_{D}(\tilde X,\tilde X\setminus \exc).}}}
\]

We study the morphism $\tau^* \colon H^{\bullet}_{D}(X,X\setminus Z)\to H^{\bullet}_{D}(\tilde X,\tilde X\setminus \exc)$ more closely. First, it is injective by \cite[Theorem 3.2]{tardini} and  Lemma \ref{lemalg}
shows that $\tau^{*}:H^h_{dR}(X)\ra H^h_{dR}(\tilde X)$ is injective (in fact this is already implied by \cite[Theorem 3.1]{wells}) and that $j^{*}$ in the second row induces an \iso

\begin{equation}\label{dR1}
H^{h}_{D}(\tilde X,\tilde X \setminus \exc)/\tau^{*}H^{h}_{D}(X,X\setminus Z)\overset\sim\lra 
H^h_{dR}(\tilde X)/\tau^{*}H^h_{dR}(X).
\end{equation}
We try to express the left hand side in terms of the cohomologies of $Z$ and $\exc$ and along the way
we reprove the injectivity of $\tau^{*}$ on the relative cohomology (cf. Remark \ref{rmk:assumptions-special-cases}.\,(1) below).

Let $\pi:N:=N_{Z|X}\rightarrow Z$ be the normal bundle of $Z$ in $X$.
Recall that $\exc$ is  the projectivization $\mathbb{P}(N)$ of $N$ and that $\tau_{\exc}:=\tau|_{\exc} \colon \exc=\mathbb{P}(N)\rightarrow Z$ is the projection of the bundle.
The normal bundle of $\exc$ in $\tilde X$ is the tautological bundle $\tilde\pi \colon T\rightarrow \exc=\mathbb{P}(N)$. It is a subbundle of $\tau_\exc^{*}N$ with the universal bundle $Q$ as the quotient so that we have an exact sequence of vector bundles on $\exc$ (cf. \eqref{exactuniv0}):
\begin{equation}\label{exactuniv}
0\lra T\overset\iota\lra \tau_\exc^{*}N\lra  Q\lra 0.
\end{equation}
Recall that 
$\tau_\exc^*N=\{\,(\nu,e)\in N\times \exc\mid \pi(\nu)=\tau_\exc(e)\,\}$
so that we have the commutative diagram
\[
\SelectTips{cm}{}
\xymatrix
@R=.7cm
{  \exc\ar[d]_-{\tau_\exc}& \tau_\exc^{*}N\ar[l]_{\varpi}\ar[d]^-{p}\\
Z & N\ar[l]_{\pi},}
\]
where $p$ and $\varpi$ denote the restrictions of the projections onto the first and the second factors, respectively.

Let $\varphi:U\overset\sim\rightarrow W$ be a diffeomorphism as given by the Tubular Neighbourhood Theorem, with $U$ and $W$ neighbourhoods of $Z$
in $X$ and $N$, respectively. We set $r=\pi\circ\varphi:U\rightarrow Z$. We may choose neighbourhoods $\tilde U$ and $\tilde W$ of $\exc$ in $\tilde X$ and $T$ and a diffeomorphism $\tilde\varphi:\tilde U\overset\sim\rightarrow \tilde W$ so that $\tilde U=\tau^{-1}U$ and $\varphi\circ\tau\circ(\tilde\varphi)^{-1}:\tilde W\rightarrow W$ is equal to $p\circ\iota|_{\tilde W}$.
We set $\tilde r=\tilde\pi\circ\tilde\varphi:\tilde U\rightarrow \exc$ so that we have the commutative diagram
\[
\SelectTips{cm}{}
\xymatrix
@R=.7cm
{  \exc\ar[d]_-{\tau_\exc}& \tilde U\ar[l]_{\tilde r}\ar[d]^-{\tau|_{\tilde U}}\\
Z & U\ar[l]_{r}.}
\]

We have the Thom class $\Psi_{Z}\in H^{2k}_{D}(X,X\setminus Z)= H^{2k}_{D}(U,U\setminus Z)$ of $Z$ and that $\Psi_{\exc}\in H^{2}_{D}(\tilde X,\tilde X\setminus \exc)= H^{2}_{D}(\tilde U,\tilde U\setminus \exc)$ of $\exc$.

\begin{lemma}\label{thomdR} In the above situation, we have
\[
\tau^{*}\Psi_{Z}=\Psi_{\exc}\smallsmile \tilde r^{*}c^{k-1}(Q),
\]
where $c^{k-1}(Q)$ is the 
top Chern class of $Q$.
\end{lemma}
\begin{proof}[Proof of Lemma \ref{thomdR}] Noting that $r\circ\tau=\tau_\exc\circ\tilde r$, we have the exact sequence of vector bundles on $\tilde U$:
\begin{equation}\label{exact}
0\lra \tilde r^{*}T\lra\tau^{*}r^{*}N\lra \tilde r^{*} Q\lra 0.
\end{equation}
Let $s_{\Delta}$ and $\tilde s_{\Delta}$ denote the diagonal sections of $\pi^{*}N$ on $N$ and of $\tilde\pi^{*}T$
on $T$, respectively. We denote the corresponding sections of $r^{*}N$ on $U$ and of $\tilde r^{*}T$ 
on $\tilde U$ by $s$ and $\tilde s$. We claim that $\tilde s$ is mapped to $\tau^{*}s$ by the first morphism above. To see  this, first note that $s_{\Delta}(\nu)=(\nu,\nu)$, where we think of the first component as the fibre
component. The section $s$ is given by, for $x\in U$, $s(x)=\varphi(x)\in (r^{*}N)_{x}=N_{z}$, $z=r(x)=\pi\circ\varphi(x)$. On the other hand $\tilde s_{\Delta}(t)=(t,t)$ and $\tilde s$ is given by, for $\tilde x\in\tilde U$,
$\tilde s(\tilde x)=\tilde\varphi(\tilde x)\in (\tilde r^{*}T)_{\tilde x}=T_{e}$, $e=\tilde r(\tilde x)=\tilde\pi\circ\tilde\varphi(\tilde x)$. We have $\tau^{*}s(\tilde x)=s(\tau(\tilde x))=\varphi\circ\tau(\tilde x)
=p\circ\iota\circ\tilde\varphi(\tilde x)$, which proves the claim.

Recall that $\Psi_{Z}$ is the localization of $c^{k}(r^{*}N)$ by $s$ so that $\tau^{*}\Psi_{Z}$ is the localization of $c^{k}(\tau^{*}r^{*}N)$ by $\tau^{*}s$. The latter can be described as follows.
Let $\tilde\nabla_{0}$ be an $\tilde s$-trivial connection for $\tilde r^{*}T$ on $\tilde U_{0}$ and let $\nabla^{Q}$
be a connection for $Q$ on $\exc$. Then there exists a $\tau^{*}s$-trivial connection $\nabla_{0}$ for $\tau^{*}r^{*}N$ on $\tilde U_{0}$ such that $(\tilde\nabla_{0},\nabla_{0},\tilde r^{*}\nabla^{Q})$ is compatible with  \eqref{exact} on
$\tilde U_{0}$. Let $\tilde\nabla_{1}$ be an arbitrary connection for $\tilde r^{*}T$ on $\tilde U$. Then there exists a  connection $\nabla_{1}$ for $\tau^{*}r^{*}N$ on $\tilde U$ such that $(\tilde\nabla_{1},\nabla_{1},\tilde r^{*}\nabla^{Q})$ is compatible with  \eqref{exact} on
$\tilde U$. Then $\tau^{*}\Psi_{Z}$ is represented by
\[
(c^{k}(\nabla_{1}),c^{k}(\nabla_{0},\nabla_{1}))=(c^{1}(\tilde\nabla_{1})\cdot \tilde r^{*}c^{k-1}(\nabla^{Q}),
c^{1}(\tilde\nabla_{0},\tilde\nabla_{1})\cdot \tilde r^{*}c^{k-1}(\nabla^{Q}).
\]
Since $(c^{1}(\tilde\nabla_{1}),
c^{1}(\tilde\nabla_{0},\tilde\nabla_{1}))$ represents $\Psi_{\exc}$, we have the lemma.
\end{proof}

From the above lemma, we see that the following diagram is commutative:
\begin{equation}\label{cor2.11}
\SelectTips{cm}{}
\xymatrix
@R=.7cm
{H^{h-2k}_{dR}(Z) \ar[d]_-{\chi}\ar@<0.5ex>[r]^-{T_{Z}}&H^{h}_{D}(X,X\setminus Z)\ar@<0.5ex>[l]^-{r_{*}}\ar[d]^-{\tau^{*}}\\
H^{h-2}_{dR}(\exc) \ar@<0.5ex>[r]^-{T_{\exc}}&H^{h}_{D}(\tilde X,\tilde X\setminus \exc)\ar@<0.5ex>[l]^-{\tilde r_{*}},}
\end{equation}
where $\chi$ is the morphism given by $z\mapsto c^{k-1}(Q)\smallsmile \tau_\exc^{*}z$. In the above $T_{Z}$ and
$r_{*}$ are isomorphisms and the inverses of each other, similarly for $T_{\exc}$ and
$\tilde r_{*}$. 
Thus $\chi$ is injective and $T_{\exc}$ 
induces an \iso
\begin{equation}\label{dR2}
H^{h-2}_{dR}(\exc)/\chi H^{h-2k}_{dR}(Z)\overset\sim\lra  
H^{h}_{D}(\tilde X,\tilde X\setminus \exc)/\tau^{*}H^{h}_{D}(X,X\setminus Z).
\end{equation}

Now we study the left hand side. We claim that $H^{\bullet}_{dR}(\exc)$ is a free $H^{\bullet}_{dR}(Z)$-module with basis $1,\gamma,\dots,\gamma^{k-2},c^{k-1}(Q)$, $\gamma=c^{1}(T)$.
To see this,  from \eqref{exactuniv} we have the relation $c(T)\cdot c(Q)=\tau_\exc^*c(N)$ among the total Chern
classes. Thus $c(Q)=c(T)^{-1}\cdot \tau_\exc^*c(N)$ and we have 
\begin{equation}\label{taut-univ}
c^{k-1}(Q)=\sum_{i=0}^{k-2}(-1)^{i}\gamma^{i}\cdot\tau_\exc^*c^{k-1-i}(N)+(-1)^{k-1}\gamma^{k-1},
\end{equation}
which proves the claim in view of Proposition \ref{prop:deRham-projectivized}. Thus we have
\begin{equation}\label{isodRsub}
H^{h-2}_{dR}(\exc)/\chi H^{h-2k}_{dR}(Z)\simeq\bigoplus _{i=0}^{k-2}\gamma^{i}\cdot\tau_\exc^* H^{h-2i-2}_{dR}(Z)\subset H^{h-2}_{dR}(\exc).
\end{equation}

By \eqref{dR1}, \eqref{dR2} and \eqref{isodRsub}, we have the diagram:
\[
\SelectTips{cm}{}
\xymatrix
@C=.7cm
@R=.6cm
{0\ar[r]&H^{h}_{dR}(X)\ar[r]^-{\tau^{*}}&H^{h}_{dR}(\tilde X)\ar[r]^-{\pi}& H^{h}_{dR}(\tilde X)/\tau^{*}H^{h}_{dR}(X)\ar[r]&0\\
{} &{}&{} & \bigoplus _{i=0}^{k-2}\gamma^{i}\cdot\tau_\exc^* H^{h-2i-2}_{dR}(Z)\ar[u]_-{\wr}\ar[ul]_{\eta}.}
\]
The restriction of the Gysin morphism $(i_{\exc})_{*}=j^{*}\circ T_{\exc}:H^{h-2}_{dR}(\exc)\ra H^{h}_{dR}(\tilde X)$ gives a splitting $\eta$ and we have:
\begin{proposition}\label{isodR} There is an \iso
\[
H^{h}_{dR}(X)\oplus\bigoplus_{i=0}^{k-2}H^{h-2i-2}_{dR}(Z)\overset\sim\lra H^{h}_{dR}(\tilde X),
\]
which is given by $(x,(z_{i})_{i=0}^{k-2})\mapsto \tau^{*}x+\sum_{i=0}^{k-2}(i_{\exc})_{*}(\gamma^{i}\cdot\tau_{\exc}^{*}z_{i})$ for $x\in H^{h}_{dR}(X)$ and $z_{i}\in H^{h-2i-2}_{dR}(Z)$.
\end{proposition}

\medskip
\noindent{\itshape Dolbeault cohomology.}
Considering the  exact sequence \eqref{lexcd} for the pairs $(X,X \setminus Z)$ and $(\tilde X, \tilde X \setminus \exc)$,
we have the commutative diagram with exact rows:
\[
{\resizebox{\textwidth}{!}{
\SelectTips{cm}{}
\xymatrix
{H^{p,q-1}_{\delbar}(X\setminus Z)\ar[r]^-{\delta}\ar[d]^{\tau^*}_{\wr} & H^{p,q}_{\bar D}(X,X\setminus Z)\ar[r] ^-{j^*}\ar[d]^{\tau^*} &H^{p,q}_{\delbar}(X)\ar[r]^-{i^{*}}\ar[d]^{\tau^*} & H^{p,q}_{\delbar}(X\setminus Z)\ar[r]^-{\delta}\ar[d]^-{\tau^*}_{\wr} & H^{p,q+1}_{\bar D}(X,X\setminus Z)\ar[d]^{\tau^*} \\
H^{p,q-1}_{\delbar}(\tilde X \setminus \exc)\ar[r]^-{\delta} & H^{p,q}_{\bar D}(\tilde X,\tilde X \setminus \exc)\ar[r]^-{j^*} & H^{p,q}_{\delbar}(\tilde X)\ar[r]^-{i^{*}}& H^{p,q}_{\delbar}(\tilde X \setminus \exc)\ar[r]^-{\delta} & H^{p,q+1}_{\bar D}(\tilde X,\tilde X\setminus \exc).}}}
\]

The essential difference from the de~Rham case occurs for the relative cohomology and 
 the morphism $\tau^{*}:H^{p,q}_{\bar D}(X,X\setminus Z)\rightarrow H^{p,q}_{\bar D}(\tilde X,\tilde X\setminus E)$, which  we are going to analyze. First, it is injective by \cite[Theorem 3.1]{tardini} and  Lemma \ref{lemalg}
shows that $\tau^{*}:H^{p,q}_{\delbar}(X)\ra H^{p,q}_{\delbar}(\tilde X)$ is injective (again this is already implied by \cite[Theorem 3.1]{wells}) and that $j^{*}$ in the second row induces an \iso

\begin{equation}\label{Dol1}
H^{p,q}_{\bar D}(\tilde X,\tilde X \setminus \exc)/\tau^{*}H^{p,q}_{\bar D}(X,X\setminus Z)\overset\sim\lra 
H^{p,q}_{\delbar}(\tilde X)/\tau^{*}H^{p,q}_{\delbar}(X).
\end{equation}
We try to express the left hand side in terms of  cohomologies of $Z$ and $\exc$.

Recall that the normal bundle $\pi:N\rightarrow Z$ of $Z$ is a holomorphic vector bundle
of rank $k$ on $Z$.  By the assumption \eqref{htnbd}, we see that
 there exist neighbourhoods $U$ and $W$ of $Z$ in $X$ and
$N$, respectively, and a biholomorphic map $\varphi:U\rightarrow W$ so that $r=\pi\circ\varphi:U\to Z$.
Thus we have isomorphisms
\[
H^{p,q}_{\bar D}(X,X\setminus Z)\simeq H^{p,q}_{\bar D}(U,U\setminus Z)\underset{\varphi^{*}}{\overset\sim\longleftarrow} H^{p,q}_{\bar D}(W,W\setminus Z)\simeq H^{p,q}_{\bar D}(N,N\setminus Z),
\]
where the first and the last isomorphisms are excisions. The $\delbar$-Thom class $\bar\Psi_{Z}$ of $Z$ is, by definition, the class in $H^{k,k}_{\bar D}(X,X\setminus Z)$ that corresponds to $\bar\Psi_{N}$ by the above isomorphism. 
We have the $\delbar$-Thom morphism
\[
\bar T_{Z}:H^{p-k,q-k}_{\delbar}(Z)\lra H^{p,q}_{\bar D}(U,U\setminus Z)=H^{p,q}_{\bar D}(X,X\setminus Z),
\]
which is given by $z\mapsto\bar\Psi_{Z}\smallsmile r^{*}z$. It gives a splitting of
\[
0\lra\operatorname{ker}\bar r_{*}\lra H^{p,q}_{\bar D}(X,X\setminus Z)\overset{\bar r_{*}}\lra H^{p-k,q-k}_{\delbar}(Z)\lra 0.
\]

Under the assumption \eqref{htnbd}, $\exc$ also admits a holomorphic retraction $\tilde r\colon \tilde U\rightarrow \exc$, $\tilde U=\tau^{-1}U$, such that the following diagram is commutative:
\[
\SelectTips{cm}{}
\xymatrix
@R=.7cm
{\exc\ar[d]_-{\tau_\exc}&\tilde U\ar[l]_{\tilde r}\ar[d]^-{\tau|_{\tilde U}}\\
Z&U\ar[l]_{r}.}
 \]
Thus we have 
 the $\delbar$-Thom class  $\bar\Psi_{\exc}\in H^{1,1}_{\bar D}(\tilde X,\tilde X\setminus \exc)= H^{1,1}_{\bar D}(\tilde U,\tilde U\setminus \exc)$ of $\exc$ and the $\delbar$-Thom morphism
\[
\bar T_{\exc}:H^{p-1,q-1}_{\delbar}(\exc)\lra H^{p,q}_{\bar D}(\tilde U,\tilde U\setminus\tilde  Z)=H^{p,q}_{\bar D}(\tilde X,\tilde X\setminus \exc),
\]
which is given by $a\mapsto\bar\Psi_{\exc}\smallsmile \tilde r^{*}a$. It gives a splitting of
\[
0\lra \operatorname{ker}\bar{\tilde  r}_{*}\lra H^{p,q}_{\bar D}(\tilde X,\tilde X\setminus \exc)\overset{\bar{\tilde  r}_{*}}\lra H^{p-1,q-1}_{\delbar}(\exc)\lra0.
\]

We have the following lemma, which is the holomorphic analogue of Lemma~\ref{thomdR} and is proven by the same argument with de Rham cohomology and Chern classes are replaced by Dolbeault cohomology and Atiyah classes, respectively:
\begin{lemma}\label{thpbD} We have:
\[
\tau^{*}\bar\Psi_{Z}=\bar\Psi_{\exc}\smallsmile \tilde r^{*}a^{k-1}(Q),
\]
where $a^{k-1}(Q)$ denotes
the top Atiyah class of $Q$.
\end{lemma} 

From the above lemma, we see that the following diagrams are commutative:
\[
\SelectTips{cm}{}
\xymatrix
@R=.7cm
{H^{p-k,q-k}_{\delbar}(Z)\ar[r]^-{\bar T_{ Z}} \ar[d]_-{\bar\chi}&H^{p,q}_{\bar D}(X,X\setminus Z)\ar[d]^-{\tau^{*}}\\
H^{p-1,q-1}_{\delbar}(\exc) \ar[r]^-{\bar T_{ \exc}}&H^{p,q}_{\bar D}(\tilde X,\tilde X\setminus \exc),}\qquad
 \xymatrix
@R=.7cm
{H^{p-k,q-k}_{\delbar}(Z) \ar[r]^-{\bar T_{ Z}}\ar[d]_-{\bar\chi}&H^{p,q}_{\bar D}(X,X\setminus Z)\ar[d]^-{\tau^{*}}\\
H^{p-1,q-1}_{\delbar}(\exc)& H^{p,q}_{\bar D}(\tilde X,\tilde X\setminus \exc)\ar[l]_{\bar{\tilde r}_{*}},}
\]
where $\bar\chi$ is the morphism given by  $z\mapsto a^{k-1}(Q)\smallsmile\tau_{\exc}^{*}z$. 

From the first commutative diagram above, $\bar T_{\tilde Z}$ induces a
well-defined morphism
\[
\psi:H^{p-1,q-1}_{\delbar}(\exc)/\bar\chi H^{p-k,q-k}_{\delbar}(Z)\lra 
H^{p,q}_{\bar D}(\tilde X,\tilde X\setminus \exc)/\tau^{*}H^{p,q}_{\bar D}(X,X\setminus Z).
\]

\begin{proposition}\label{propaim} Under the assumption \eqref{eq:technical},
$\psi$ is an \iso.
\end{proposition}
\begin{proof}[Proof of Proposition \ref{propaim}] For the surjectivity, take $[\tilde c]$, $\tilde c \in H^{p,q}_{\bar D}(\tilde X,\tilde X\setminus \exc)$.
We may write $\tilde c=\bar T_{\tilde Z}(\tilde a)+\rho$, $\rho\in \op{ker}\bar{\tilde r}_{*}$. Thus
by the second condition in \eqref{eq:technical}, $\psi([\tilde a])=[\tilde c]$. For the injectivity, take $[\tilde a]$ \st\ $T_{\tilde Z}(\tilde a)=
\tau^{*}(c)$ for some $c\in H^{p,q}_{\bar D}(X,X\setminus Z)$. Then $\tilde a=\bar{\tilde r}_{*}\circ T_{\tilde Z}(\tilde a)=\bar{\tilde r}_{*}\circ\tau^{*}(c)$ and by the first condition in \eqref{eq:technical}, $\tilde a\in\op{im}\bar\chi$.
\end{proof}

As in the case of de Rham, we see that $H^{\bullet,\bullet}_{\delbar}(\exc)$ is a free $H^{\bullet,\bullet}_{\delbar}(Z)$-module with basis $1,\alpha,\dots,\alpha^{k-2},a^{k-1}(Q)$, $\alpha=a^{1}(T)$.
Thus we have
\begin{equation}\label{isoDolsub}
H^{p-1,q-1}_{\delbar}(\exc)/\bar\chi H^{p-k,q-k}_{\delbar}(Z)\simeq\bigoplus _{i=0}^{k-2}\alpha^{i}\cdot\tau_\exc^* H^{p-i-1,q-i-1}_{\delbar}(Z)\subset H^{p-1,q-1}_{\delbar}(\exc).
\end{equation}

Under the assumption \eqref{eq:technical},
the restriction of the $\delbar$-Gysin morphism $(\bar i_{\exc})_{*}=j^{*}\circ \bar T_{\exc}:H^{p-1.q-1}_{\delbar}(\exc)\ra H^{p,q}_{\delbar}(\tilde X)$ gives a splitting $\eta$: 
\[
\SelectTips{cm}{}
\xymatrix
@C=.7cm
@R=.6cm
{0\ar[r]&H^{p,q}_{\delbar}(X)\ar[r]^-{\tau^{*}}&H^{p,q}_{\delbar}(\tilde X)\ar[r]^-{\pi}& H^{p,q}_{\delbar}(\tilde X)/\tau^{*}H^{p,q}_{\delbar}(X)\ar[r]&0\\
{} &{}&{} & \bigoplus _{i=0}^{k-2}\alpha^{i}\cdot\tau_\exc^* H^{p-i-1,q-i-1}_{\delbar}(Z)\ar[u]_-{\wr}\ar[ul]_{\eta}.}
\]
and we have:
\begin{proposition}\label{isoDol} Under the assumption \eqref{eq:technical}, there is an \iso
\[
H^{p,q}_{\delbar}(X)\oplus\bigoplus_{i=0}^{k-2}H^{p-i-1,q-i-1}_{\delbar}(Z)\overset\sim\lra H^{p,q}_{\delbar}(\tilde X),
\]
which is given by $(x,(z_{i})_{i=0}^{k-2})\mapsto \tau^{*}x+\sum_{i=0}^{k-2}(\bar i_{\exc})_{*}(\alpha^{i}\cdot\tau_{\exc}^{*}z_{i})$ for $x\in H^{p,q}_{\delbar}(X)$ and $z_{i}\in H^{p-i-1,q-i-1}_{\delbar}(Z)$.
\end{proposition}

The theorem follows from  Propositions \ref{isodR} and \ref{isoDol} (cf. Definition \ref{defH}.\,1).
\end{proof}

\begin{remark}\label{rmk:assumptions-special-cases}
(1) Even if $X$ and $Z$ admit a natural Hodge structure, {\itshape i.e.}, satisfy the $\del\delbar$-Lemma, it is not clear, from the above arguments, whether or not $\tilde X$ has the same property. The problem is that the cohomology of $\exc$ contributes to the cohomology of $\tilde X$ through the Gysin morphisms and it is not clear if these morphisms send good representatives to good ones as in (H1) and (H2) of Definition \ref{defH}.\,2.

\smallskip

\noindent
(2) In view of the commutative diagram \eqref{cor2.11}, which is a consequence of Lemma \ref{thomdR}, the 
injectivity of $\tau^{*}$ on the relative cohomology is equivalent to that of $\chi$. 
From the definition of $\chi$, we see that this is also equivalent to the injectivity of $\tau_{\exc}^{*}$. The injectivity of $\chi$ 
can also be proven as follows, independently of the injectivity of $\tau^{*}$.

Recalling that $\tau_\exc \colon \exc=\mathbb{P}(N)\rightarrow Z$ is a $\mathbb{P}^{k-1}$-bundle, we have the integration along the fibres 
$(\tau_\exc)_{*} \colon H^{h-2}_{dR}(\exc)\rightarrow H^{h-2k}_{dR}(Z)$. 
First we claim that, for $(\tau_\exc)_* \colon H^{2k-2}_{dR}(\exc)\rightarrow H^{0}_{dR}(Z)$, we have
\begin{equation}\label{prop2.12}
(\tau_\exc)_*c^{k-1}(Q)=1.
\end{equation}
This can be seen from \eqref{taut-univ}, 
the projection formula
and the facts that 
$(\tau_\exc)_*\gamma^{i}=0$, for $i=0,\dots,k-2$, by dimension
reason, and 
$(\tau_\exc)_*\gamma^{k-1}=(-1)^{k-1}$, as $\gamma$ restricted to each fibre is the
first Chern class of the tautological bundle 
on $\mathbb{P}^{k-1}$.
Then by the projection formula and  \eqref{prop2.12}, 
\[
(\tau_\exc)_*\circ\chi(a)=(\tau_\exc)_*(c^{k-1}(Q)\smallsmile \tau_\exc^*(a))=(\tau_\exc)_*c^{k-1}(Q)\smallsmile a=a.
\]
Thus 
the composition $(\tau_\exc)_*\circ\chi$ is the identity morphism of $H^{h-2k}_{dR}(Z)$ so that
$\chi$
is injective. Thus $\tau_{\exc}^{*}$ is also injective. If $\exc$ is K\"ahler, this follows from  \cite[Theorem 4.1]{wells}.
\smallskip

\noindent
(3) The statement of Proposition \ref{isodR}  is proven in the K\"ahler context, {\itshape e.g.} in \cite[Theorem 7.31]{voisin-1} by excision and by the Thom isomorphism in cohomology with $\Z$-coefficients. Presumably, the K\"ahler condition is necessary there to show that $\chi$ or $\tau_{\exc}^{*}$ is injective using 
the above-mentioned theorem \cite[Theorem 4.1]{wells}.
The novelty here is the elimination of this restriction  by a  result of \cite{tardini} or Lemma~\ref{thomdR}, which
also gives a precise relation between the Thom classes of $Z$ and $\exc$
and this in turn gives a precise relation between $\tau^{*}$ and $\chi$.

This subject is treated in the algebraic category in \cite[\S \ 6.7]{fulton}.
\smallskip

\noindent
(4) In the Dolbeault case,
we can show the injectivity of $\bar\chi$  similarly as for $\chi$ (cf. (1) above). However this does not directly imply
the injectivity of $\tau^{*}$ on the relative cohomology.  The injectivity of $\bar\chi$ is equivalent to that of $\tau_{\exc}^{*} \colon H^{p-k,q-k}_{\delbar}(Z)\rightarrow H^{p-k,q-k}_{\delbar}(\exc)$.
If $\exc$ is K\"ahler, the latter again follows from  \cite[Theorem 4.1]{wells}.
 \smallskip

\noindent
(5) The condition regarding the holomorphically-contractible neighbourhood of $Z$ in Theorem \ref{thm:blowup} holds, for example, if $Z$ is a point (see Example \ref{example:blowup-point}), or if $X$ is a fibration ({\itshape e.g.} a Hopf manifold) with $Z$ a fibre.
 \smallskip

\noindent
(6) The first condition in \eqref{eq:technical} is implied by
 the commutativity of the diagram \eqref{diagfund},
which may be verified for the top degree cohomology using the projection formula.
\end{remark}

\begin{example}[{Blow-up in a point; see also \cite[Proposition 3.6]{yang-yang}}]\label{example:blowup-point}
The very particular case when $Z$ is a point is easier, and follows by the description of the Dolbeault cohomology in \cite{griffiths-harris}. For completeness we outline the proof in this situation.
Let $X$ be a compact complex manifold and consider $\tau \colon \tilde X \to X$ the blow-up of $X$ on a point $p$. If $X$ admits a Hodge structure, then also $\tilde X$ does.

We denote by $\exc=\mathbb{P}^{n-1}=\tau^{-1}(p)$ the exceptional divisor
of the blow-up.
We recall that the de Rham and Dolbeault cohomologies of $X$ and $\tilde X$ are related as follows (see \cite[pages 473-474]{griffiths-harris}): for $k\notin\left\lbrace 0\,,\,2n\right\rbrace$, for $(p,q)\notin\left\lbrace (0,0)\,,\,(n,n)\right\rbrace$,
$$
H^\bullet(\tilde X,\mathbb{C})=H^\bullet(X,\mathbb{C})\oplus H^\bullet(\exc,\mathbb{C})
$$
and
$$
H^{\bullet,\bullet}_{\overline\partial}(\tilde X)=H^{\bullet,\bullet}_{\overline\partial}(X)\oplus H^{\bullet,\bullet}_{\overline\partial}(\exc).
$$
In particular, $h^{p,p}(\tilde X)=h^{p,p}(X)+1$ and $h^{p,q}(\tilde X)=
h^{p,q}(X)$ for $p\neq q$.
Since, by hypothesis, $X$ satisfies the $\partial\overline\partial$-lemma and $\exc$ clearly does, we have that
\begin{eqnarray*}
H^k(\tilde X,\mathbb{C}) &=& H^k(X,\mathbb{C})\oplus H^k(\exc,\mathbb{C}) \\
&=& \bigoplus_{p+q=k}H^{p,q}_{\overline\partial}(X) \oplus \bigoplus_{r+s=k}H^{r,s}_{\overline\partial}(\exc) \\
&=& \bigoplus_{t+v=k}\left(H^{t,v}_{\overline\partial}(X)\oplus H^{t,v}_{\overline\partial}(\exc)\right) = \bigoplus_{t+v=k}H^{t,v}_{\overline\partial}(\tilde X)
\end{eqnarray*}
and
\begin{eqnarray*}
\overline{H^{p,q}_{\overline\partial}(\tilde X)} &=& \overline{H^{p,q}_{\overline\partial}(X)\oplus H^{p,q}_{\overline\partial}(\exc)} \\
&=& H^{q,p}_{\overline\partial}(X)\oplus H^{q,p}_{\overline\partial}(\exc) = H^{q,p}_{\overline\partial}(\tilde X) .
\end{eqnarray*}
\end{example}

\begin{question}\label{question:submanifold}
We ask whether {\itshape a submanifold of a manifold satisfying the $\partial\overline\partial$-Lemma, still satisfies the $\partial\overline\partial$-Lemma}. Note that, in general, existence of Hodge structures is not preserved by blow-ups: Claire Voisin suggested to us an example that appears in \cite{vuli-cras} by Victor Vuletescu: take the blow-up of a Hopf surface inside $\mathbb{S}^3\times\mathbb{S}^3\times\mathbb{P}^1$.
(Compare also \cite[Concluding Remarks]{yang-yang}.)
\end{question}

\begin{question}\label{question:normal-cone}
We ask whether {\itshape if $X$ and $Z$ satisfy the $\partial\overline\partial$-Lemma, then we can perform constructions like the deformation to the normal cone for $(X,Z)$ that still satisfies the $\partial\overline\partial$-Lemma}. We recall that the deformation to the normal cone by MacPherson \cite[Chapter 5]{fulton} allows to modify the pair $(X,Z)$ to the pair $(N_{Z|X},Z)$ as deformation, where clearly $Z$ has the property of admitting a holomorphically contractible neighbourhood in its normal bundle $N_{Z|X}$. We briefly recall the construction, see also \cite[Section 8]{suwa-ASPM}: consider a $1$-dimensional disc $\mathbb D$; define $X^*:=\mathrm{Bl}_{Z\times\{0\}}(X\times\mathbb{D}) \setminus \mathrm{Bl}_{Z\times\{0\}}(X \times\{0\})$ that provides a deformation path through $X^*_t=X$ to $X^*_0=N_{Z|X}$. We notice that $N_{Z|X}$ is clearly non-compact. We also recall that satisfying the $\partial\overline\partial$-Lemma is an open property under deformations \cite[Proposition 9.21]{voisin-1}, but in general it is not closed \cite{angella-kasuya-NWEJM}.
\end{question}

\begin{remark}\label{remark:corollary}
{\itshape If Questions \ref{question:normal-cone} and \ref{question:submanifold} have positive answers}, {\em if we can avoid the technical assumption \eqref{eq:technical}}, and {\em if we can prove the naturality of the induced Hodge structures on the blow-up}, then our argument would give that the $\partial\overline\partial$-Lemma property is defined inside the localization of the category of holomorphic maps with respect to bimeromorphisms, equivalently, modifications. More precisely:
{\itshape Let $f \colon M \to N$ be a bimeromorphic map between compact complex manifolds of the same dimension. Then $M$ satisfies the $\partial\overline\partial$-Lemma if and only if $N$ does.}
(The same would be true assuming $M$ K\"ahler without Question \ref{question:submanifold}.)
Indeed, this will follow by the Weak Factorization Theorem for bimeromorphic maps between compact complex manifolds \cite[Theorem 0.3.1]{akmw}, \cite{wlodarczyk}.
It states that $f$ can be functorially factored as a sequence of blow-ups and blow-downs with non-singular centres.
For a blow-up $\varphi\colon V' \to V''$, we have that: if $V'$ satisfies the $\partial\overline\partial$-Lemma, then $V''$ does by \cite[Theorem 5.22]{deligne-griffiths-morgan-sullivan}; if $V''$ satisfies the $\partial\overline\partial$-Lemma, then $V'$ does by Theorem \ref{thm:blowup}.
Compare \cite[Question 1.2]{rao-yang-yang} and \cite[Corollary 28]{stelzig-doublecomplex}.
\end{remark}

\section{The orbifold case}
We now consider the orbifold case, applying the Stelzig arguments to the orbifold Dolbeault cohomology studied in \cite{baily-PNAS, baily-AJM}.
Recall that an \emph{orbifold}, also called \emph{V-manifold} \cite{satake}, is a singular complex space whose singularities are locally isomorphic to quotient singularities $\mathbb C^n \slash G$, where $G\subset \mathrm{GL}(n, \mathbb C)$ is a finite subgroup. Tensors on an orbifold are defined to be locally $G$-invariant. In particular, this yields the notions of orbifold de Rham cohomology and orbifold Dolbeault cohomology, for which we have both a sheaf-theoretic and an analytic interpretation \cite{satake, baily-PNAS, baily-AJM}, and Hodge decomposition in cohomology defines the orbifold $\del\delbar$-Lemma property.

The following result generalizes the contents of Theorem \ref{thm:stelzig} to orbifolds {\em of global-quotient type}, namely, $X \slash G$, where $X$ is a complex manifold and $G$ is a finite group of biholomorphisms of $X$.
We can interpret this case as the smooth case with the further action of a group $G$: for example, an orbifold morphism $Z \slash H \to X \slash G$ is just an equivariant map $Z \to X$. The orbifold Dolbeault cohomology of $X \slash G$ is the cohomology of the complex of $G$-invariant forms, $\left( (\wedge^{\bullet,\bullet} X)^G, \overline\partial \right)$. The notion of $\del\delbar$-Lemma for orbifolds refers to the cohomological decomposition for the double complex $\left( (\wedge^{\bullet,\bullet} X)^G, \del, \delbar \right)$.
This result follows directly by the work of Jonas Stelzig and it will let us construct new examples of compact complex manifolds satisfying the $\partial\overline\partial$-Lemma, as resolutions of orbifolds obtained starting from compact quotients of solvable Lie groups.
(Here, by asking that $j^o \colon Z^o = Z / G \to X^o = X \slash G$ is a {\em suborbifold}, we mean that $Z$ is a $G$-invariant submanifold of $X$, and the embedding $j \colon Z \to X$ is $G$-equivariant.)

\begin{theorem}[see \cite{stelzig-doublecomplex}]\label{thm:orbifolds}
Let $X^o = X \slash G$ be a compact complex orbifold of complex dimension $n$, and $j^o \colon Z^o = Z / G \to X^o$ be a suborbifold of complex dimension $d$ and codimension $k:=n-d$, and consider $\tau^o \colon \tilde X_{Z^o} \to X^o$ the blow-up of $X^o$ along the centre $Z^o$.
If both $X^o$ and $Z^o$ satisfy the $\partial\overline\partial$-Lemma, then also $\tilde X_{Z^o}$ does satisfy the $\partial\overline\partial$-Lemma.
\end{theorem}

\begin{proof}
We first notice that $\tilde X_{Z^o}$ itself is a (possibly smooth) orbifold of global-quotient type. Indeed, by the universal property of blow-up, see {\itshape e.g.} \cite[page 604]{griffiths-harris}, the action $G \circlearrowleft X$ yields the action $G \circlearrowleft \tilde X_Z$ the blow-up of $X$ along $Z$.
The proof then follows by considering the {\em $E_1$-quasi-isomorphism} $\wedge^{\bullet,\bullet} \tilde X_Z \simeq_1 \wedge^{\bullet,\bullet} X \oplus \bigoplus_{j=1}^{k-1} \wedge^{\bullet-j,\bullet-j} Z$. This means that there is a morphism of double complexes that induces an isomorphism at the first page $E_1$ of the Fr\"olicher spectral sequence, that is, the Dolbeault cohomology, see \cite[Definition D]{stelzig-doublecomplex}. The fact that there is an $E_1$-quasi-isomorphism as above is \cite[Theorem 23]{stelzig-doublecomplex}, see also \cite{stelzig-blowup}.
Since the action of $G$ is compatible with the above morphism, we get also an $E_1$-quasi-isomorphism $(\wedge^{\bullet,\bullet} \tilde X_Z)^G \simeq_1 (\wedge^{\bullet,\bullet} X)^G \oplus \bigoplus_{j=1}^{k-1} (\wedge^{\bullet-j,\bullet-j} Z)^G$.
Recall that the Dolbeault and the de Rham cohomologies of the orbifold are computed as the cohomologies of the complex of $G$-invariant forms, as said above. Therefore the properties of Hodge decomposition for $X$ and $Z$ reflects on the property of Hodge decomposition for $\tilde X_Z$ by means of the above quasi-isomorphism.
\end{proof}

\begin{example}[resolution of an orbifold covered by the Iwasawa manifold]\label{example:global-quotient-iwasawa}
In this example, starting from a smooth compact complex manifold which does not satisfy the $\del\delbar$-lemma, we construct a simply-connected smooth compact complex manifold that does.

Consider the complex Heisenberg group
\[
G:=\left\lbrace
\left(\begin{matrix}
1 & z_1 & z_3\\
0 & 1 & z_2\\
0 & 0 &1
\end{matrix}\right)
\;:\; z_1,z_2,z_3\in\mathbb{C}\right\rbrace.
\]
It is a nilpotent Lie group, and it is endowed with a bi-invariant complex structure defined by the coframe of $(1,0)$-forms
$$ \varphi^1:=dz_1 , \quad \varphi^2:=dz_2 , \quad \varphi^3:=dz_3-z_1dz_2.$$
They have structure equations
$$ d \varphi^1 = 0 , \quad d\varphi^2 = 0 , \quad d\varphi^3 = - \varphi^1 \wedge \varphi^2 . $$

Let $\xi\neq1$ be a cubic root of the unity, and $\Lambda$ be the lattice generated by $1$ and $\xi$. Consider the subgroup $\Gamma$ in $G$
consisting of matrices with entries in $\Lambda$.
The compact quotient $M:=G/\Gamma$ is a holomorphically-parallelizable nilmanifold \cite{nakamura}. By \cite{nomizu, sakane}, the de Rham and Dolbeault cohomologies of $M$ are the same as the cohomologies of the Iwasawa manifold, which are computed for instance in \cite{schweitzer}.

We consider the following action of the finite group $\mathbb{Z}_3$ on $G$:
$$
\sigma \colon (z_1,z_2,z_3) \mapsto (\xi z_1,\xi z_2, \xi^2 z_3).
$$
It is easy to check that the action is linear, and since $\xi^2=-1-\xi$ then $\Gamma$ is $\sigma$-invariant. Therefore we get an action on the quotient $M$, and a complex space $M^o:=M/\left\langle\sigma\right\rangle$
with orbifold singularities.
The action on the global co-frame
of $(1,0)$-forms becomes
$$
\sigma^*(\varphi^1)=\xi\varphi^1,\quad
\sigma^*(\varphi^2)=\xi\varphi^2,\quad
\sigma^*(\varphi^3)=\xi^2\varphi^3.
$$

We compute the orbifold de Rham and Dolbeault cohomologies by taking the $\sigma$-invariant forms.
We have
\begin{eqnarray*}
\lefteqn{ \wedge^\bullet M^o= (\wedge^\bullet M)^{\langle \sigma\rangle} } \\
&=& \wedge \langle 1, \varphi^{13},
\varphi^{23},\varphi^{1\bar 1},\varphi^{1\bar 2},\varphi^{2\bar 1},
\varphi^{2\bar 2}, \varphi^{3\bar3}, \varphi^{\bar 1\bar 3},\varphi^{\bar 2\bar 3} , \varphi^{12\bar3}, \varphi^{3\bar1\bar2} \rangle
\end{eqnarray*}
as an algebra, with the only non-trivial differentials
$$ d \varphi^{3\bar3} = \varphi^{12\bar3}-\varphi^{3\bar1\bar2}, \quad d \varphi^{12\bar3} = d \varphi^{3\bar1\bar2} = \varphi^{12\bar1\bar2}. $$
It is straightforward to check that this complex satisfies the $\del\delbar$-Lemma, that is, the orbifold $M^o$ satisfies the
$\del\delbar$-lemma.

Now we resolve the singularities of $M^o$ in order to obtain a simply-connected
smooth compact complex manifold satisfying the $\del\delbar$-lemma.
The procedure is similar to the one described in \cite{fernandez-munoz}.
The singular locus of $M^o$ consists of $3^3$ isolated singular points.
By blowing-up each point, we get an exceptional divisor $\mathbb{CP}^2 / \mathbb Z_3$, where the action is given by
$$ \sigma \colon [z_1:z_2:z_3] \mapsto [z_1:z_2:\xi z_3] . $$
The singular locus consists now of the isolated point $q:=[0:0:1]$ and of the complex projective line $L:=\left\lbrace[z_1:z_2:0]\right\rbrace\subset \mathbb{C}\mathbb{P}^2$, which both admit a holomorphically contractible neighbourhood. Finally, by blowing-up the $q$'s and the $L$'s, we get a smooth model $\tilde M$. Thanks to Theorem \ref{thm:orbifolds}, the performed operations mantain the $\del\delbar$-Lemma property.

In fact, the same argument as \cite[Proposition 2.3]{fernandez-munoz} adapted to our manifold $M$, which is a principal $2$-torus bundle over a $4$-torus, yields that $\tilde M$ is simply-connected.
Moreover, the metric
$$ \omega := \frac{\sqrt{-1}}{2} \sum_{j=1}^3\varphi^j\wedge\bar\varphi^j $$
on $M$ is $\sigma$-invariant and so it descends to the orbifold $M^o$.
We can also obtain $\tilde M$ by blowing-up $M$ and then by quotienting by $\mathbb Z_3$. Therefore, $\omega$ yields a balanced metric on $\tilde M$ thanks to \cite{alessandrini-bassanelli}.

Finally, we notice that $\tilde M$ is not in class $\mathcal C$ of Fujiki, since $M$ is not.
\end{example}

Summarizing the contents of the last example:
\begin{theorem}\label{thm:examples-simply-connected}
There exists a simply-connected compact complex non-K\"ahler manifold $\tilde M$ such that: it is non-K\"ahler, in fact it does not belong to class $\mathcal C$ of Fujiki; it satisfies the $\del\delbar$-Lemma; and it is endowed with a balanced metric.
\end{theorem}


\begin{thebibliography}{48}

\bibitem[ABST13]{abate-bracci-suwa-tovena}
M. Abate, F. Bracci, T. Suwa, F. Tovena, Localization of Atiyah classes, {\em Rev. Mat. Iberoam.} \textbf{29} (2013), 547--578.

\bibitem[AKMW02]{akmw}
D. Abramovich, K. Karu, K. Matsuki, J. W\l odarczyk, Torification and factorization of birational maps, {\em J. Amer. Math. Soc.} \textbf{15} (2002), no. 3, 531--572.

\bibitem[AB96]{alessandrini-bassanelli}
L. Alessandrini, G. Bassanelli, The class of compact balanced manifolds is invariant under modifications, {\em Complex analysis and geometry (Trento, 1993)}, 1--17, Lecture Notes in Pure and Appl. Math., \textbf{173}, Dekker, New York, 1996.

\bibitem[AK17a]{angella-kasuya-AGAG}
D. Angella, H. Kasuya, Bott-Chern cohomology of solvmanifolds, {\em Ann. Global Anal. Geom.} \textbf{52} (2017), no. 4, 363--411.

\bibitem[AK17b]{angella-kasuya-NWEJM}
D. Angella, H. Kasuya, Cohomologies of deformations of solvmanifolds and closedness of some properties, {\em North-West. Eur. J. Math.} \textbf{3} (2017), 75--105.

\bibitem[AT17]{angella-tardini} D. Angella, N. Tardini,
Quantitative and qualitative cohomological properties for non-K\"ahler
manifolds. {\em Proc. Amer. Math. Soc.} \textbf{145} (2017) no.~1, 273--285.

\bibitem[AT13]{angella-tomassini-3}
D. Angella, A. Tomassini, On the $\partial\overline\partial$-Lemma and Bott-Chern cohomology, {\em Invent. Math.} \textbf{192} (2013), no.~1, 71--81.

\bibitem[Bai54]{baily-PNAS}
W. L. Baily, On the quotient of an analytic manifold by a group of analytic homeomorphisms, {\em Proc. Nat. Acad. Sci. U. S. A.} \textbf{40} (1954), no. 9, 804--808.

\bibitem[Bai56]{baily-AJM}
W. L. Baily, The decomposition theorem for $V$-manifolds, {\em Amer. J. Math.} \textbf{78} (1956), no. 4, 862--888.

\bibitem[BFM14]{bazzoni-fernandez-munoz}
G. Bazzoni, M. Fern\'andez, V. Mu\~noz, A $6$-dimensional simply connected complex and symplectic manifold with no K\"ahler metric, {\em J. Symplectic Geom.} \textbf{16} (2018), no. 4, 1001--1020.

\bibitem[Bla56]{blanchard}
A. Blanchard,
Sur les vari\'et\'es analytiques complexes,
{\em Ann. Sci.  l'\'E.N.S.} \textbf{73} (1956), no 2, 157--202.

\bibitem[Buc99]{buchdahl}
N. Buchdahl, On compact K\"ahler surfaces, {\em Ann. Inst. Fourier (Grenoble)} \textbf{49} (1999), no. 1, vii, xi, 287--302.

\bibitem[Cam91]{campana}
F. Campana, The class $\mathcal C$ is not stable by small deformations, {\em Math. Ann.} \textbf{290} (1991), no. 1,
19--30.

\bibitem[COUV16]{couv}
M. Ceballos, A. Otal, L. Ugarte, R. Villacampa, Invariant complex structures on $6$-nilmanifolds: classification, Fr\"olicher spectral sequence and special Hermitian metrics, {\em J. Geom. Anal.} \textbf{26} (2016), no. 1, 252--286.

\bibitem[CFGU00]{cordero-fernandez-gray-ugarte}
L.A. Cordero, M. Fern\'andez, A. Gray and L. Ugarte,
Compact nilmanifolds with nilpotent complex structures: Dolbeault cohomology,  {\em Trans. Amer. Math. Soc.} \textbf{352} (2000), no. 12, 5405--5433.

\bibitem[DGMS75]{deligne-griffiths-morgan-sullivan}
P. Deligne, Ph.~A. Griffiths, J. Morgan, D.~P. Sullivan, Real homotopy theory of K\"ahler manifolds, {\em Invent. Math.} \textbf{29} (1975), no.~3, 245--274.

\bibitem[FM08]{fernandez-munoz}
M. Fern\'andez, V. Mu\~noz,
An $8$-dimensional nonformal, simply connected, symplectic manifold,
{\em Ann. of Math. (2)} \textbf{167} (2008), no. 3, 1045--1054.

\bibitem[Fri91]{friedman}
R. Friedman, On threefolds with trivial canonical bundle, {\em Complex geometry and Lie theory (Sundance, UT, 1989)}, 103--134,  Proc. Sympos. Pure Math., 53, Amer. Math. Soc., Providence, RI, 1991.

\bibitem[Fri17]{friedman-arxiv1708}
R. Friedman, The $\partial\overline\partial$-lemma for general Clemens manifolds, \texttt{arXiv:1708.00828}	.

\bibitem[Fr\"o55]{frolicher}
A. Fr\"olicher, Relations between the cohomology groups of Dolbeault and topological invariants, {\em Proc. Nat. Acad. Sci.
U.S.A.} \textbf{41} (1955), 641--644.

\bibitem[Ful84]{fulton}
W. Fulton, {\em Intersection theory}, Second edition, Ergebnisse der Mathematik und ihrer Grenzgebiete, \textbf{3}, Folge, A Series of Modern Surveys in Mathematics, \textbf{2}, Springer-Verlag, Berlin, 1998. 

\bibitem[GH78]{griffiths-harris}
Ph. Griffiths, J. Harris, {\em Principles of algebraic geometry}, Reprint of the 1978 original, Wiley Classics Library, John Wiley \& Sons, Inc., New York, 1994.

\bibitem[Has89]{hasegawa}
K. Hasegawa, Minimal models of nilmanifolds, {\em Proc. Am. Math. Soc.} \textbf{106} (1989), no. 1, 65--71.

\bibitem[Hir62]{hironaka}
H. Hironaka, An example of a non-K\"ahlerian complex-analytic deformation of K\"ahlerian complex structures, {\em Ann. Math. (2)} \textbf{75} (1962), no. 1, 190--208.

\bibitem[HIS18]{honda-izawa-suwa}
N. Honda, T. Izawa, T. Suwa, Sato hyperfunctions via relative Dolbeault cohomology, \texttt{arXiv:1807.01831 [math.CV]}, 2018

\bibitem[Kas13b]{kasuya-JGP}
H. Kasuya, Hodge symmetry and decomposition on non-K\"ahler solvmanifolds, {\em J. Geom. Phys.} \textbf{76} (2014), 61--65. 

\bibitem[Lam99]{lamari}
A. Lamari, Courants k\"ahl\'eriens et surfaces compactes, {\em Ann. Inst. Fourier (Grenoble)} \textbf{49} (1999), no. 1, vii, x, 263--285.

\bibitem[LP92]{lebrun-poon}
C. LeBrun, Y.S. Poon, Twistors, K\"ahler manifolds, and bimeromorphic geometry. II, {\em J. Am. Math. Soc.} \textbf{5} (1992), no. 2, 317--325.

\bibitem[Nak75]{nakamura}
I. Nakamura, Complex parallelisable manifolds and their small deformations, {\em J. Differential Geometry} \textbf{10} (1975), 85--112.

\bibitem[Nom54]{nomizu}
K. Nomizu, On the cohomology of compact homogeneous spaces of nilpotent Lie groups, {\em Ann. of Math. (2)} \textbf{59} (1954), no.~3, 531--538.

\bibitem[Pop13]{popovici-JGA}
D. Popovici, Stability of strongly Gauduchon manifolds under modifications, {\em J. Geom. Anal.} \textbf{23} (2013), no. 2, 653--659.

\bibitem[Pop15a]{popovici-BSMF}
D. Popovici, Aeppli cohomology classes associated with Gauduchon metrics on compact complex manifolds, {\em Bull. Soc. Math. France} \textbf{143} (2015), no. 4, 763--800.

\bibitem[Pop15b]{popovici}
D. Popovici, Volume and self-intersection of differences of two nef classes, {\em Ann. Sc. Norm. Super. Pisa Cl. Sci. (5)} \textbf{17} (2017), no. 4, 1255--1299.

\bibitem[RYY17]{rao-yang-yang}
S. Rao, S. Yang, X. Yang, Dolbeault cohomologies of blowing up complex manifolds, {\em J. Math. Pures Appl. (9)} \textbf{130} (2019), 68--92.

\bibitem[Sak76]{sakane}
Y. Sakane, On compact complex parallelisable solvmanifolds, {\em Osaka J. Math.} \textbf{13} (1976), no. 1, 187--212.

\bibitem[Sat56]{satake}
I. Satake, On a generalization of the notion of manifold, {\em Proc. Nat. Acad. Sci. U.S.A.} {\bfseries 42} (1956), no.~6, 359--363.

\bibitem[Sch07]{schweitzer}
M. Schweitzer, Autour de la cohomologie de Bott-Chern, Pr\'epublication de l'Institut Fourier no.~703 (2007), \texttt{arXiv:0709.3528}.

\bibitem[Ste18a]{stelzig-thesis}
J. Stelzig, {\em Double Complexes and Mixed Hodge Structures as Vector Bundles}, PhD thesis, WWU, M\"unster, 2018.

\bibitem[Ste18b]{stelzig-blowup}
J. Stelzig, The Double Complex of a Blow-up, to appear in {\em Internat. Math. Res. Not. IMRN}, \textsc{doi}: 10.1093/imrn/rnz139\texttt{arXiv:1808.02882}.

\bibitem[Ste18c]{stelzig-doublecomplex}
J. Stelzig, On the Structure of Double Complexes, \texttt{arXiv:1812.00865}.

\bibitem[Suw98]{suwa-book}
T. Suwa, {\em Indices of vector fields and residues of singular holomorphic foliations}, Actualit\'es Math\'ematiques, Hermann, Paris, 1998.

\bibitem[Suw08]{suwa-intersection} T. Suwa, Residue theoretical approach to intersection theory, {\em Contemp. Math. Amer. Math. Soc.} \textbf{459} (2008), 207--261.

\bibitem[Suw09]{suwa-ASPM}
T. Suwa, \v{C}ech-Dolbeault cohomology and the $\overline\partial$-Thom class, {\em Singularities---Niigata---Toyama 2007}, 321--340, Adv. Stud. Pure Math., \textbf{56}, Math. Soc. Japan, Tokyo, 2009.

\bibitem[Suw19]{suwa-relD} T. Suwa, Relative Dolbeault cohomology,
\texttt{arXiv:1903.04710}.

\bibitem[Tar19]{tardini}
N. Tardini, Relative \v{C}ech-Dolbeault homology and applications, \texttt{arXiv:1812.00362}, to appear in \emph{Ann. Mat. Pura Appl}.

\bibitem[TW13]{tosatti-weinkove-Crelle}
V. Tosatti, B. Weinkove, Hermitian metrics, $(n-1, n-1)$ forms and Monge-Amp\`ere equations, {\em J. Reine Angew. Math.} \textbf{755} (2019), 67--101

\bibitem[Voi02]{voisin-1}
C. Voisin, {\em Hodge theory and complex algebraic geometry. I}, Translated from the French by Leila Schneps, Reprint of the 2002 English edition, Cambridge Studies in Advanced Mathematics, 76, Cambridge University Press, Cambridge, 2007.

\bibitem[Vul12]{vuli-cras}
V. Vuletescu, Exemples de faisceaux coh\'erents sans r\'esolution localement libre en dimension $3$, {\em C. R. Math. Acad. Sci. Paris} \textbf{350} (2012), no. 7-8, 411--412.

\bibitem[Wlo03]{wlodarczyk}
J. W\l odarczyk, Toroidal varieties and the weak factorization theorem, {\em Invent. Math.} \textbf{154} (2003), no. 2, 223--331.

\bibitem[YY17]{yang-yang}
S. Yang, X. Yang, Bott-Chern blow-up formula and bimeromorphic invariance of the $\partial\overline\partial$-Lemma for threefolds, \texttt{arXiv:1712.08901}.

\bibitem[Wel74]{wells}
R. O. Wells, Comparison of de Rham and Dolbeault cohomology for proper surjective mappings, {\em Pacific J. Math.} \textbf{53} (1974), 281--300.

\bibitem[Wu06]{wu}
C.-C. Wu, On the geometry of superstrings with torsion, Thesis (Ph.D.)Harvard University, Proquest LLC, Ann Arbor, MI, 2006.

\end{thebibliography}
\end{document}